\documentclass[12pt]{amsart} \usepackage{amssymb,amsmath,amstext}
\usepackage[colorlinks=false, urlcolor=blue, linkcolor=blue,
citecolor=blue]{hyperref} \usepackage[margin=1in]{geometry}
\usepackage {tikz}
\usepackage {cite}

\usetikzlibrary{arrows,decorations.pathmorphing,backgrounds,positioning,fit,matrix}
\definecolor{processblue}{cmyk}{0.96,0,0,0} \usepackage{enumerate}
\input{xy} \xyoption{all} \usepackage{xypic}

\numberwithin{equation}{section}

\theoremstyle{plain} \newtheorem{thm}{Theorem}[section]
\newtheorem{prop}[thm]{Proposition} \newtheorem{lem}[thm]{Lemma}

\theoremstyle{definition} \newtheorem{dfn}[thm]{Definition}
\newtheorem{exmp}[thm]{Example} \newtheorem{rem}[thm]{Remark}

\DeclareMathOperator{\supp}{Supp} 
\begin{document}
\author{Nasrin Altafi} \address{Department of Mathematics, KTH Royal
  Institute of Technology, S-100 44 Stockholm, Sweden}
\email{nasrinar@kth.se} \author{Mats Boij} \address{Department of
  Mathematics, KTH Royal Institute of Technology, S-100 44 Stockholm,
  Sweden} \email{boij@kth.se}

\title{The weak Lefschetz property of equigenerated monomial ideals}

\begin{abstract}
  We determine a sharp lower bound for the Hilbert function in degree
  $d$ of a monomial algebra failing the weak Lefschetz property over a polynomial ring
  with $n$ variables and generated in degree $d$, for any $d\geq 2$
  and $n\geq 3$. We consider artinian ideals in the polynomial ring
  with $n$ variables generated by homogeneous polynomials of degree
  $d$ invariant under an action of the cyclic group
  $\mathbb{Z}/d\mathbb{Z}$, for any $n\geq 3$ and any $d\geq 2$. We
  give a complete classification of such ideals in terms of the weak Lefschetz property
  depending on the action.

\end{abstract}
\maketitle
\section{Introduction}
The weak Lefschetz property (WLP) for an artinian graded algebra $A$
over a field $\mathbb{K}$, says there exists a linear form $\ell$ that
induces, for each degree $i$, a multiplication map
$\times \ell : (A)_i\longrightarrow (A)_{i+1}$ that has maximal rank,
i.e. that is either injective or surjective. Though many algebras are
expected to have the WLP, establishing this property for a specific
class of algebras is often rather difficult. In this paper we study
the WLP of the specific class of algebras which are the quotients of a
polynomial ring $S=\mathbb{K}[x_1,\dots , x_n]$ over field
$\mathbb{K}$ of characteristic zero by artinian monomial ideals
generated in the same degree $d$.  For this class of artinian
algebras, E. Mezzetti and R. M. Mir\'o-Roig \cite{MM}, showed that
$2n-1$ is the sharp lower bound for the number of generators of $I$
when the injectivity fails for $S/I$ in degree $d-1$. In fact they
give the lower bound for the number of generators for the minimal
monomial Togliatti systems $I\subset \mathbb{K}[x_1,\dots , x_n]$ of
the forms of degree $d$.  For more details see the original articles
of Togliatti \cite{Togliatti2,Togliatti1}.  In the first part
of this article we establish the lower bound for the number of
monomials in the cobasis of the ideal $I$ in the ring $S$ or
equivalently, lower bound for the Hilbert function of $S/I$ in degree
$d$, which is $H_{S/I}(d) :=\dim_{\mathbb{K}} ({S/I})_d$, where
sujectivity fails in degree $d-1$. Observe that once multiplication
by a general linear form on a quotient of $S$ is surjective, then it
remains surjective in the next degrees. This implies that all these
algebras with the Hilbert function $H_{S/I}(d)$ below our bound
satisfy the WLP.

In the main theorems of the first part of this paper, we provide a
sharp lower bound for $H_{S/I}(d)$ for artinian monomial algebra
$S/I$, where the surjectivity fails for $S/I$ in degree $d-1$. For the
cases when the number of variables is less than three the bound is
known. The first main theorem provides the bound when the polynomial
ring has three variables.
\begin{thm}
  Let $I\subset S=\mathbb{K}[x_1,x_2,x_3]$ be an artinian monomial
  ideal generated in degree $d$, for $d\geq 2$ such that $S/I$ fails
  to have the WLP. Then we have that
  \[ H_{S/I}(d) \geq \left\{
    \begin{array}{ll}
      3d-3& \mbox{if $d$ is odd }\\
      3d-2 & \mbox{if $d$ is
             even}.
    \end{array} \right. 
  \]
  Furthermore, the bounds are sharp.
\end{thm}

In the second theorem we provide a sharp bound when the
number of variables is more than three.
\begin{thm}
  Let $I\subset S=\mathbb{K}[x_1,\dots ,x_n]$ be an artinian monomial
  ideal generated in degree $d$, for $d\geq 2$ and $n\geq 4$ such that
  $S/I$ fails to have the WLP. Then we have that
$$
H_{S/I}(d)\geq 2d.
$$
Furthermore, the bound is sharp.
\end{thm}
In \cite{MMO}, Mezzetti, Mir\'o-Roig and Ottaviani describe a connection
between projective varieties satisfying at least one Laplace equation
and homogeneous artinian ideals generated by polynomials of the same
degree $d$ failing the WLP by failing injectivity of the multiplication
map by a linear form in degree $d-1$. In \cite{MM2}, Mezzetti and Mir\'o-Roig construct a class of examples
of Togliatti systems in three variables of any degree. More precisely,
they consider the action on $S=\mathbb{K}[x,y,z]$ of cyclic group
$\mathbb{Z}/d\mathbb{Z}$ defined by
$[x,y,z]\mapsto [\xi^ax,\xi^by,\xi^cz]$, where $\xi$ is a primitive
$d$-th root of unity and $\gcd(a,b,c,d)=1$. They prove that the ideals
generated by forms of degree $d$ invariant by such actions are all
defined by monomial Togliatti systems of degree $d$. In \cite{MM3}, Colarte, Mezzetti, Mir\'o-Roig and Salat show that in $S=\mathbb{K}[x_1,\dots,x_n]$ the ideal fixed by the action of cyclic group $\mathbb{Z}/d\mathbb{Z}$, defined by $[x_1,\dots ,x_n]\mapsto [\xi^{a_1}x_1,\dots , \xi^{a_n}x_n]$, where $\gcd(a_1,\dots ,a_n,d)=1$ and there are at most $\binom{n+d-2}{n-2}$ fixed monomials is a monomial Togliatti system.

In this article, we generalize the result in \cite{MM3} and in Theorem
\ref{sur}, we prove that these ideals satisfy the WLP if and only if
at least $n-1$ of the integers $a_i$ are equal. In addition, in the
polynomial ring with three variables we give a formula for the number
of fixed monomials and we provide bounds for such numbers.

\section{Preliminaries}
We consider standard graded algebras $S/I$, where $S = \mathbb{K}[x_1,\dots , x_n]$, $I$ is a homogeneous ideal of $S$, $\mathbb{K}$ is a field of characteristic zero and the $x_i$'s all have degree $1$. Our ideal $I$ will be  an artinian monomial ideal generated in a single degree $d$. Given a polynomial $f$ we denote the set of monomials with non-zero coefficients in $f$ by $\supp(f)$. \\
Now let us define the weak and strong Lefschetz properties for
artinian algebras.
\begin{dfn}
  Let $I\subset S$ be a homogeneous artinian ideal. We say that $S/I$
  has the \textit{Weak Lefschetz Property} (WLP) if there is a linear
  form $\ell\in (S/I)_1$ such that, for all integers $j$, the
  multiplication map
$$
\times \ell : (S/I)_j\longrightarrow (S/I)_{j+1}
$$
has maximal rank, i.e. it is injective or surjective. In this case the linear form $\ell$ is called a \textit{Lefschetz element} of $S/I$. If for general linear form $\ell\in (S/I)_1$ and for an integer $j$ the map $\times \ell$ does not have the maximal rank we will say that $S/I$ fails the WLP in degree $j$. \\
We say that $S/I$ has the \textit{Strong Lefschetz Property} (SLP) if
there is a linear form $\ell\in(S/I)_1$ such that, for all integers
$j$ and $k$ the multiplication map
$$
\times \ell^k : (S/I)_j\longrightarrow (S/I)_{j+k}
$$ has maximal rank, i.e. it is injective or surjective. We often abuse the notation and say that $I$ fails or satisfies the WLP or SLP, when we mean that $S/I$ does so.
\end{dfn}
In the case of one variable, the WLP and SLP hold trivially since all
ideals are principal. Harima, Migliore, Nagel and Watanabe in
\cite[Proposition $4.4$]{HMNW}, proved the following result in two
variables.
\begin{prop}\label{SLPcodim2}
  Every artinian ideal in $\mathbb{K}[x,y]$ (char $\mathbb{K}=0$) has
  the Strong Lefschetz property (and consequently also the Weak
  Lefschetz property).
\end{prop}

In a polynomial ring with more than two variables, it is not true in
general that every artinian monomial algebra has the SLP or WLP. Also
it is often rather difficult to determine whether a given algebra
satisfies the SLP or even WLP. One of the main general results in a
ring with more than two variables is proved by Stanley in
\cite{Stanley}.
\begin{thm}\label{CI}
  Let $S=\mathbb{K}[x_1 , \dots ,x_n]$, where
  $char(\mathbb{K})=0$. Let $I$ be an artinian monomial complete
  intersection, i.e $I=(x^{a_1}_1,\dots , x^{a_n}_n)$. Then $S/I$ has
  the SLP.
\end{thm}

Because of the action of the torus $(\mathbb{K}^*)^n$ on monomial
algebras, there is a canonical linear form that we have to
consider. In fact we have the following result in \cite[Proposition $2.2$]{MMN},
 proved by Migliore, Mir\'o-Roig and Nagel.
\begin{prop}\label{mon} Let $I\subset S$ be an artinian monomial
  ideal. Then $S/I$ has the weak Lefschetz property if and only if
  $x_1+x_2+\cdots +x_n$ is a weak Lefschetz element for $S/I$.
\end{prop}

Let us now recall some facts of the theory of the \textit{inverse
  system}, or \textit{Macaulay duality}, which will be a fundamental
tool in this paper. For a complete introduction, we refer the reader
to \cite{Geramita} and \cite{IK}.

Let $R= \mathbb{K}[y_1,\dots ,y_n]$, and consider $R$ as a graded $S$-module where the action of $x_i$ on $R$ is partial differentiation with respect to $y_i$.\\
There is a one-to-one correspondence between graded artinian algebras
$S/I$ and finitely generated graded $S$-submodules $M$ of $R$, where
$I=Ann_S(M)$ and is the annihilator of $M$ in $S$ and, conversely,
$M=I^{-1}$ is the $S$-submodule of $R$ which is annihilated by $I$
(cf. \cite[Remark 1]{Geramita}), p.17).  Since the map
$\circ \ell : R_{i+1}\longrightarrow R_i$ is dual of the map
$\times \ell : (S/I)_{i}\longrightarrow (S/I)_{i+1}$ we conclude that
the injectivity (resp. surjectivity) of the first map is equivalent to
the surjectivity (resp. injectivity) of the second one. Here by
$"\circ \ell"$ we mean that the linear form $\ell$ acts on $R$.

For a monomial ideal $I$ the inverse system module $(I^{-1})_d$ is
generated by the corresponding monomials of $S_d$ but not in $I_d$ in
the dual ring $R_d$.

Mezzetti, Mir\'o-Roig and Ottaviani in \cite{MMO} describe a relation
between existence of artinian ideals $I\subset S$ generated by
homogeneous forms of degree $d$ failing the WLP and the existence of
projections of the Veronese variety
$V(n-1,d)\subset \mathbb{P}^{\binom{n+d-1}{d}-1}$ satisfying at least one Laplace equation of
order $d-1$.

For an artinian ideal $I\subset S$, they make the following construction. Assume that $I$ is minimally generated by the homogeneous polynomials  $f_1,\dots ,f_r$ of degree $d$ and denote by $I^{-1}$, the inverse system module of $I$. Since $I$ is artinian, the polynomials $f_1,\dots ,f_r$ define a regular morphism
$$
\varphi_{I_d}:\mathbb{P}^{n-1}\longrightarrow \mathbb{P}^{r-1}. 
$$
Denote by $X_{n-1,I_d}$, the closure of the image of ${\varphi}_{I_d}$. 
There is a rational map 
$$
\varphi_{(I^{-1})_d} :\mathbb{P}^{n-1}\dashrightarrow
\mathbb{P}^{\binom{n+d-1}{d}-r-1}
$$
associated to $(I^{-1})_d$. Denote by $X_{n-1,(I^{-1})_d}$, the closure of the image of $\varphi_{(I^{-1})_d}$.

With notations as above, in \cite{MMO}, Theorem $3.2$, Mezzetti,
Mir\'o-Roig and Ottaviani prove the following theorem.
\begin{thm}
  Let $I\subset S$ be an artinian ideal generated by $r$ forms
  $f_1,\dots ,f_r$ of degree $d$. If $r\leq \binom{n+d-2}{n-2}$, then
  the following conditions are equivalent:
  \begin{itemize}
  \item[$(1)$] The ideal $I$ fails the WLP in degree $d-1$,
  \item[$(2)$] The forms $f_1,\dots ,f_r$ become $\mathbb{K}$-linearly
    dependent on a general hyperplane $H$ of $\mathbb{P}^{n-1}$,
  \item[$(3)$]The $n-1$-dimensional variety $X_{n-1,(I^{-1})_d}$
    satisfies at least one Laplace equation of order $d-1$.
  \end{itemize}
\end{thm}
If $I$ satisfies the three equivalent conditions in the above theorem,
$I$ (or $I^{-1}$) is called a \textit{Togliatti system}.

\section{On the Support of form $f$ annihilated by $\ell$ and its higher powers }
Let $S= \mathbb{K}[x_1,\dots ,x_n]$ be the polynomial ring where
$n\geq 3$ and $\mathbb{K}$ is a field of characteristic zero. In this
section we give some definitions and notations and prove some results
about the number of monomials in the support of polynomials
$f\in (I^{-1})_d$ with $(x_1+\cdots +x_n)^a\circ f=0$ for some
$1\leq a\leq d$.  Now let us define a specific type of well known
integer matrices which we use them throughout this section.
\begin{dfn}
  For a non-negative integer $k$ and positive integer $m$, where
  $k\leq m$, we define the Toeplitz matrix $T_{k,m}$, to be the
  following $(k+1)\times (m+1)$ matrix
$$T_{k,m}=
\begin{bmatrix}
  \binom{m-k}{0} & \binom{m-k}{1}   &  \binom{m-k}{2} & \cdots & \binom{m-k}{m-k}& 0& \cdots  & 0 \\\\
  0 &  \binom{m-k}{0}   &  \binom{m-k}{1}& \cdots  & \binom{m-k}{m-k-1}& \binom{m-k}{m-k} &\cdots  & 0 \\\\
  \vdots&  \vdots &  \vdots & \vdots & \vdots  &\vdots&\vdots &\vdots  \\\\

  0 & 0 &0 & \cdots &\binom{m-k}{m-k-3} & \binom{m-k}{m-k-2} &
  \binom{m-k}{m-k-1} & \binom{m-k}{m-k}
\end{bmatrix}
$$
where the $({i,j})^{th}$ entry of this matrix is $\binom{m-k}{j-i}$
and we use the convention that $\binom{m}{i}=0$ if $i<0$ or $m>i$.
\end{dfn}
We have the following useful lemma which proves the maximal minors of
$T_{k,m}$ are non-zero.
\begin{lem}\label{Toeplitz}
  For each non-negative integer $k$ and positive integer $m$ where
  $k\leq m$, all maximal minors of the Toeplitz matrix $T_{k,m}$ are
  non-zero.
\end{lem}
\begin{proof}
  Let $R = \mathbb{K}[x,y]$ be the polynomial ring in variables $x$
  and $y$ and choose monomial bases
  $\mathcal{A}:=\lbrace x^jy^{k-j}\rbrace^{k}_{j=0}$ and
  $\mathcal{B}:=\lbrace x^iy^{m-i}\rbrace^{m}_{i=1}$ for the
  $\mathbb{K}$-vector spaces $R_k$ and $R_m$, respectively. Observe
  that $T_{k,m}$ is the matrix representing the multiplication map
  $\times (x+y)^{m-k} : R_k\rightarrow R_m$ with respect to the bases
  $\mathcal{A}$ and $\mathcal{B}$.  Given any square submatrix $M$ of
  size $k+1$, define ideal $J\subset R$ generated by the subset of
  monomials in $\mathcal{B}$, called $\mathcal{B}^\prime$,
  corresponding to the columns of $T_{k,m}$ not in $M$. Therefore,
  $\mathcal{A}$ and $\mathcal{B}\setminus \mathcal{B}^\prime$ form
  monomial bases for $(R/J)_k$ and $(R/J)_m$, respectively and $M$ is
  the matrix representing the multiplication map
  $\times (x+y)^{m-k} : (R/J)_{k} \rightarrow (R/J)_{m}$ with respect
  to $\mathcal{A}$ and $\mathcal{B}\setminus \mathcal{B}^\prime$.
  Since by Proposition \ref{SLPcodim2}, any monomial $R$-algebra has
  the SLP, and by Proposition \ref{mon} $x+y$ is a Lefschetz element
  for $R/J$, the multiplication map by $x+y$ is a bijection and
  therefore the matrix $M$ has non-zero determinant. This implies that
  all the maximal minors of $T_{k,m}$ are non-zero.
\end{proof}
Consider a non-zero homogeneous polynomial $f$ of degree $d$ in the
dual ring $R=\mathbb{K}[y_1,\dots ,\allowbreak y_n]$ where we have
$(x_1+\cdots +x_n)\circ f=0$. We use the following notations and
definitions to prove some properties of such polynomial $f$.
\begin{dfn}\label{def:nu} For an ideal $I$ of $S$, we denote the
  Hilbert function of $S/I$ in degree $d$ by
  $H_{S/I}(d) := \dim_\mathbb{K}(S/I)_d$, and the set of all artinian
  monomial ideals of $S$ generated in a single degree $d$ by
  $\mathcal{I}_d$. In addition, for an artinian ideal $I$ we define
  $\phi (I,d) : \times (x_1+\dots+x_n):(S/I)_{d-1} \rightarrow
  (S/I)_{d}$ and
$$
\nu(n,d): = min \lbrace H_{(S/I)}(d) \mid \phi(I,d) \hspace*{0.1
  cm}\text{is not surjective, for}\hspace*{1 mm}
I\in\mathcal{I}_d\rbrace.
$$
\end{dfn}
\begin{dfn}\label{def:deg/level}
  In a polynomial ring $S=\mathbb{K}[x_1,\dots , x_n]$, for any
  monomial $m$ and variable $x_i$, we define
  $$\deg_i(m):= max \lbrace e \mid {x^e_i}\vert m \rbrace $$Define the
  set $\mathcal{M}_d$ to be the set of monomials of degree $d$ in $R$
  and denote the set of monomials of degree $k$ with respect to the
  variable $y_i$ by,
  \begin{equation*} {\mathcal{L}}^k_{i,d} := \lbrace m\in
    \mathcal{M}_d \mid \deg_i(m)=k \rbrace \subset \mathcal{M}_d.
  \end{equation*}
\end{dfn}
\begin{lem}\label{SLP}
  Let $f$ be a form of degree $d\geq 2$ in the dual ring
  $R=\mathbb{K}[y_1,\dots ,y_n]$ of $S=\mathbb{K}[x_1, \dots ,x_n]$
  and let the linear forms $\ell:=x_1+\cdots +x_n$ and
  $\ell^\prime := \ell-x_j$ for $1\leq j\leq n$.  Write
  $f=\sum^d_{i=0}y^i_jg_i$, where $g_i$ is a polynomial of degree
  $d-i$ in the variables different from $y_j$, then for every
  $0\leq c\leq d$, we have
  \begin{equation}\label{equal}
    \ell^c\circ f= \sum^{d-c}_{k=0}\sum^{c}_{i=0}\dfrac{(k+c-i)!}{k!}\binom{c}{i}y^k_j{\ell^\prime}^i\circ g_{k+c-i}.
  \end{equation}
  In particular, $\ell^c\circ f=0$ if and only if,
  \begin{equation}\label{polys}
    \sum^{c}_{i=0}\dfrac{(k+c-i)!}{k!}\binom{c}{i}{\ell^\prime}^i\circ g_{k+c-i}=0, \hspace*{ 1 cm} 0\leq k\leq d-c.
  \end{equation}

\end{lem}
\begin{proof}
  We prove the lemma using induction on $c$.
  For $c=0$ the equality (\ref{equal}) is trivial. For $c=1$, we have
  \begin{equation}\label{equal 1}
    \ell\circ f =  \sum_{k=0}^{d} {ky_j^{k-1} g_k+y_j^k \ell^{\prime}\circ g_k}
    = \sum_{k=0}^{d-1} {(k+1)y_j^k g_{k+1} + y_j^k \ell^{\prime}\circ g_k}.
  \end{equation}
  Assume the equality holds for $c-1$
  then we have $\ell^c\circ f = \ell\circ (\ell^{c-1}\circ f)$ and
  \begin{align*}
    \ell\circ (\ell^{c-1}\circ f) &= \ell\circ \left(\sum^{d-c+1}_{k=0}\sum^{c-1}_{i=0}\dfrac{(k+c-1-i)!}{k!}\binom{c-1}{i}y^k_j{\ell^\prime}^i\circ g_{k+c-1-i}\right)\\
                                  &= \sum^{d-c+1}_{k=0}\sum^{c-1}_{i=0}\dfrac{(k+c-1-i)!}{k!}\binom{c-1}{i}\left( ky^{k-1}_j{\ell^\prime}^i\circ g_{k+c-1-i}+y^{k}_j{\ell^\prime}^{i+1} g_{k+c-1-i}\right)\\
                                  &= \sum^{d-c}_{k=0}\sum^{c-1}_{i=0}\dfrac{(k+c-i)!}{(k+1)!}\binom{c-1}{i}(k+1)y^{k}_j{\ell^\prime}^i\circ g_{k+c-i}\\ &\hspace*{1.5 cm}+\dfrac{(k+c-1-i)!}{k!}\binom{c-1}{i}y^{k}_j{\ell^\prime}^{i+1}\circ  g_{k+c-1-i}\\
                                  &= \sum^{d-c}_{k=0}\sum^{c}_{i=0}\left(\dfrac{(k+1)(k+c-i)!}{(k+1)!}\binom{c-1}{i}+\dfrac{(k+c-i)!}{k!}\binom{c-1}{i-1}\right) y^k_j{\ell^\prime}^i\circ g_{k+c-i}\\
                                  & = \sum^{d-c}_{k=0}\sum^{c}_{i=0}\dfrac{(k+c-i)!}{k!}\binom{c}{i}y^k_j{\ell^\prime}^i\circ g_{k+c-i}.
  \end{align*}
\end{proof}
Using the above lemma the following proposition gives properties about
the form $f$.
\begin{prop}\label{prop:levels}
  Let $f$
  be a non-zero form of degree $d$
  in the dual ring $R=\mathbb{K}[y_1,
  \dots , y_n]$ of $S=\mathbb{K}[x_1, \dots ,
  x_n]$ such that $({x_1}+\dots + {x_n})\circ f
  =0$. Then the following conditions hold:
  \begin{itemize}
  \item[(i)] If $y^d_i\notin
    \supp(f)$, then the sum of the coefficients of
    $f$
    corresponding to the monomials in ${\mathcal{L}}^k_{i,d}
    \cap \supp(f)$ is zero; for each $0\leq k \leq d-1$.
  \item[(ii)] If $a=max\lbrace
    \deg_i(m)\mid m\in
    \supp(f)\rbrace$, then ${\mathcal{L}}^k_{i,d}\cap \supp(f)\neq
    \emptyset$; for all $0\leq k\leq a$.
  \end{itemize}
\end{prop}
\begin{proof}
  Write the form $f$
  as, $f
  = \sum_{k=0}^{d} {y_i^{k}g_k}$, where $g_k$ is a degree
  $d-k$
  polynomial in variables different from $y_i$.
  Denote $\ell
  =x_1+\dots + x_n$ and $\ell ^{\prime}=\ell - x_i$.  Since $\ell\circ
  f=0$, Lemma \ref{SLP} implies that
  \begin{equation}\label{Eq:2'} {(k+1)y_i^k g_{k+1} + y_i^k
      \ell^{\prime}\circ g_k} = 0 , \hspace*{1 cm} \forall
    \hspace*{0.2 cm}0\leq k \leq d-1.
  \end{equation}
  To show $(i)$
  we act each equation by $(\ell^{\prime})^{d-k-1}$ and we get that
  \begin{equation}\label{Eq:2}
    (k+1)(d-k-1)! g_{k+1}(1,\dots , 1) + (d-k)! g_k(1,\dots , 1) =0 \hspace*{ 1 cm} \forall \hspace*{0.2 cm}0\leq k \leq d-1.
  \end{equation}
  Since we assumed $g_{d}
  =0$ we get that $g_k(1,\dots , 1)= 0$ for all $0\leq k\leq
  d-1$, which implies that for all $0\leq k\leq
  d-1$ sum of the coefficients of
  $f$
  corresponding to the monomials in ${\mathcal{L}}^d_{i,d}\cap
  \supp(f)$ is zero and proves part $(i)$.

  To show part $(ii)$, note that
  $a=max\lbrace \deg_i(m)\mid m\in \supp(f)\rbrace$ implies that
  $g_a\neq 0$. Using Equation (\ref{Eq:2'}) recursively we get that
  $g_{j}\neq 0$ for all $0\leq j\leq a$, which means that
  ${\mathcal{L}}^j_{i,d}\cap \supp(f)\neq \emptyset$ for all
  $0\leq j\leq a$.
\end{proof}

In the following theorem we provide a bound for the number of
monomials with non-zero coefficients in the non-zero form in the
kernel of the map
$\circ (x_1+\cdots +x_n)^{d-a}\colon (I^{-1})_d\longrightarrow
(I^{-1})_{a}$.
In particular it provides a bound on the number of generators for an
equigenerated monomial ideal in $S$ failing the WLP.
\begin{thm}\label{a+2}
  Let $f\neq 0$ be a form of degree $d$ in the dual ring
  $R=\mathbb{K}[y_1,\dots ,y_n]$ of the ring
  $S=\mathbb{K}[x_1,\dots ,x_n]$. If for the linear form
  $\ell:=x_1+\cdots +x_n$ we have $\ell^{d-a}\circ f=0$ for some
  $0\leq a \leq d-1$, then $\vert \supp(f)\vert\geq a+2$.
\end{thm}
\begin{proof}

  For a variable $y_j$ write $f=\sum^{d}_{i=0}y^i_jg_i$ such that
  $g_i$ is a polynomial of degree $d-i$ in the variables different
  from $y_j$. Since for some $1\leq a\leq d-1$ we have
  $\ell^{d-a}\circ f=0$ from Lemma \ref{SLP} we have that
  \begin{equation}\label{origin}
    \sum^{d-a}_{i=0}\dfrac{(k+d-a-i)!}{k!}\binom{d-a}{i}{\ell^\prime}^i\circ g_{k+d-a-i}=0, \hspace*{ 1 cm} 0\leq k\leq a.
  \end{equation}
  For every $j$ with $1\leq j\leq a+1$ we act each equation in the above
  system by $(\ell^\prime)^{j-k-1}$, so we have
  \begin{equation}\label{prev}
    \sum^{d-a}_{i=0}\dfrac{(k+d-a-i)!}{k!}\binom{d-a}{i}{\ell^\prime}^{i+j-k-1}\circ g_{k+d-a-i}=0, \hspace*{ 1 cm} 0\leq k\leq a,\hspace*{0.5 cm}
  \end{equation}
  equivalently for each $j$ with $1\leq j\leq a+1$ we have that
  \begin{equation}\label{T}
    \sum^{d-k}_{i=a-k}\dfrac{(d-i)!}{k!}\binom{d-a}{i-(a-k)}{\ell^\prime}^{i+j-(a+1)}\circ g_{d-i}=0, \hspace*{ 1 cm} 0\leq k\leq j-1. \hspace*{0.5 cm}
  \end{equation}
  Note that for $k\geq j$ the equations in (\ref{prev}) are zero.

  For each $0\leq j\leq a+1$ the coefficient matrix of the system in
  (\ref{T}) in the forms
  $(d-i)!{\ell^\prime}^{i+j-(a+1)}\circ g_{d-i}$ is the Toeplitz
  matrix $T_{(j-1)\times(d-a+j-1)}$ up to multiplication of $k$-th row
  by $\frac{1}{k!}$. Using Lemma \ref{Toeplitz} we get that all the
  maximal minors of this coefficient matrix are non-zero. This implies
  that in each system of equations either all the terms are zero or
  there are at least $j+1$ non-zero terms.
  \\Now we want to prove the statement by induction on the number of variables $n$. Suppose $n=2$ then each $g_i$ is a monomial of degree $d-i$ in one variable. In (\ref{T}) consider the corresponding system of equations for $j=a+1$. If for every $0\leq i \leq d$ we have that ${\ell^\prime}^{i}\circ g_{d-i}=0$ implies that for every $0\leq i\leq d$ we have $g_{d-i}=0$ which contradicts the assumption that $f\neq 0$. Therefore for at least $a+2$ indices $0\leq i \leq d$ we have ${\ell^\prime}^{i}\circ g_{d-i}\neq 0$ which means $\vert \supp(f)\vert \geq a+2$.\\
  Now we assume that the statement is true for the forms $f$ in
  polynomial rings with $n-1$ ($n\geq 3$) variables and we prove it
  for the form with $n$ variables.  \\We divide it into two cases,
  suppose in the system of equations for every $1\leq j\leq a+1$ all
  terms are zero. In this case for each $1\leq j\leq a+1$, letting
  $i=a-j+1$ implies that
  $(\ell^\prime)^{a-j+1+j-(a+1)}\circ g_{d-(a-j+1)}=g_{d-a+j-1}=0$ for
  all $1\leq j\leq a+1$. Since we assume that $f\neq 0$ there exists
  $a+1\leq i\leq d$ such that $g_{d-i}\neq 0$, but considering $j=1$
  in (\ref{T}) with the assumption that all terms in this equation is
  zero we get that $(\ell^\prime)^{i-a}\circ g_{d-i}=0$. Using the
  induction hypothesis on the polynomial $g_{d-i}$ in $n-1$ variables
  we get that
  $\vert \supp(f)\vert \geq \vert \supp(g_{d-i})\vert \geq
  d-(d-i)-(i-a)+2=a+2$ as we wanted to prove.

  Now we assume that there exists $1\leq j\leq a+1$ such that there
  are at least $j+1$ indices $0\leq i\leq d$ such that
  ${\ell^\prime}^{i+j-(a+1)}\circ g_{d-i}\neq 0$ in the corresponding
  system of equations in \ref{T}. We take the largest index $j$ with
  this property and we get that for these $j+1$ indices we have that
  ${\ell^\prime}^{i+j-(a+1)+1}\circ g_{d-i}=0$. Now using the
  induction hypothesis in these polynomials we get that
  $\vert \supp(g_{d-i})\vert \geq d-(d-i)-(i+j-(a+1)+1)+2=a+2-j$,
  therefore
$$
\vert\supp(f)\vert\geq \sum^d_{i=0}\vert\supp(g_i)\vert \geq
(j+1)(a+2-j)\geq a+2.
$$
\end{proof}
\section{Bounds on the number of generators of ideals with three
  variables failing WLP }
In this section we consider artinian monomial ideals
$I\subset S=\mathbb{K}[x_1,x_2,x_3]$ generated in a single degree
$d$. In \cite{MM}, Mezzetti and Mir\'o-Roig provided a sharp lower
bound for the number of generators of such ideals failing the WLP by
failing \textit{injectivity} of the multiplication map on the algebra in degree
$d-1$. Here we prove a sharp upper bound for the number of generators
of such ideals failing the WLP by failing \textit{surjectivity} in degree $d-1$
equivalently we provide a sharp lower bound for the number of
generators of $(I^{-1})_d$ where the map
$\circ \ell\colon (I^{-1})_d\longrightarrow (I^{-1})_{d-1}$ is not
injective, where $\ell = x_1+x_2+x_3$.

First we prove an easy but interesting result. Recall that every
polynomial in at most two variables factors as a product of linear
forms over an algebraically closed field. Here we note that the same
statement holds in three variables if the polynomial vanishes by the
action of a linear form on the dual ring. This in some cases
corresponds to the failure of WLP. Note that for the WLP, the
assumption on the field to be algebraically closed is not necessary,
but in order to factor the form as a product of linear forms we need
to have this assumption on the field. In addition the statement does
not necessarily hold in polynomial rings with more than three
variables.
\begin{lem}\label{factor}
  Let $S=\mathbb{K}[x_1,x_2,x_3]$ and $S/I$ be an artinian algebra
  over an algebraically closed field $\mathbb{K}$. Let $f$ be a form
  in the kernel of the map
  $\circ \ell :(I^{-1})_i\longrightarrow (I^{-1})_{i-1}$ for a linear
  form $\ell$ and integer $i$, then $f$ factors as a product of linear
  forms each of which is annihilated by $\ell$.
\end{lem}
\begin{proof}
  By a linear change of variables we consider
  $S=\mathbb{K}[x^\prime_1,x^\prime_2,x^\prime_3]$ and
  $R=\mathbb{K}[y^\prime_1,y^\prime_2,y^\prime_3]$ simultaneously in
  such a way that $x^\prime_1=\ell$. Then we have that
  $\ell\circ f(y_1,y_2,y_3)=x^\prime_1\circ
  f(y^\prime_1,y^\prime_2,y^\prime_3)=0$
  where this implies that $f$ is a polynomial in two variables
  $y^\prime_2$ and $y^\prime_3$. Using the fact that any polynomial in
  two variables over an algebraically closed field factors as a
  product of linear forms we conclude that $f$ factors as a product of
  linear forms in $y^\prime_2$ and $y^\prime_3$. Hence all of them are
  annihilated by $\ell= x^\prime_1$.
\end{proof}
The next proposition provides a bound for the number of non-zero terms
in each homogeneous component with respect to one of the variables for
a non-zero form $f$, where $\ell \circ f=0$.
\begin{prop}\label{level mons}
Let $f$ be a non-zero form of degree
  $d\geq 2$ in the dual ring $R=\mathbb{K}[y_1,y_2,y_3]$ of
  $S=\mathbb{K}[x_1,x_2,x_3]$ such that
  $(x_1+x_2 +x_3)\circ f =0$. Then we have
  $$\vert {\mathcal{L}}^k_{i,d}\cap \supp(f) \vert \geq d-a_i+1,\hspace*{1cm} \forall\hspace*{0.2 cm} 0\leq k \leq a_i ,\hspace*{0.2 cm} 1\leq i\leq 3,$$
  where $a_i=max \lbrace \deg_i(m) \mid m \in \supp(f) \rbrace$.
\end{prop}
\begin{proof}
  Write $f = \sum_{k=0}^{a_i}{y^{k}_ig_k}$, where $g_k$ is a degree
  $d-k$ polynomial in two variables different from $y_i$. Let
  $\ell^{\prime}=\ell - x_i$, then we have
  \begin{align*}
    0=\ell\circ f = \ell\circ (\sum_{k=0}^{a_i}{y^{k}_ig_k})
    =\sum_{k=0}^{a_i}ky^{k-1}_ig_k+y^k_i\ell^\prime\circ g_k
  \end{align*}
  therefore
  \begin{equation}\label{Eq:3}
    (k+1)g_{k+1}+\ell^\prime\circ g_{k}=0, \hspace*{1cm} \forall \hspace*{0.2 cm} 0\leq k\leq a_i
  \end{equation}

  after linear change of variables to $u:=y_\alpha+y_\beta$ and
  $v:=y_\alpha-y_\beta$, Equation (\ref{Eq:3}) implies that,
  $(\partial/{\partial u})^{a_i-k+1}\circ g_k=0$ for any
  $0\leq k\leq a_i$.  Therefore, for each $0\leq k \leq a_i$ we have
  \begin{align*}
    g_k=&\sum_{j=0}^{a_i-k}{\lambda_j u^{j}v^{a_i-k+j}}=v^{d-a_i}\sum_{j=0}^{a_i-k}\lambda_j u^j v^{a_i-k+j} &\lambda_j\in \mathbb{K}.
  \end{align*}

  Rewriting $g_k$ in the variables $y_\alpha$ and $y_\beta$ we get
  that
  \begin{align*}
    g_k=&(y_\alpha-y_\beta)^{d-a_i}\sum_{j=0}^{a_i-k}\lambda_j(y_\alpha-y_\beta)^j(y_\alpha+y_\beta)^{a_i-k-j}\\=&(\sum_{s=0}^{d-a_i}(-1)^s\binom{d-a_i}{s}y^s_\alpha y^{d-a_i-s}_\beta)( \sum_{t=0}^{a_i-k}\lambda_t(y_\alpha-y_\beta)^t(y_\alpha+y_\beta)^{a_i-k-t})
  \end{align*}
  where the second sum is a polynomial of degree $a_i-k$ in the
  variables $y_\alpha$ and $y_\beta$, and since any such polynomial is
  of the form $\sum_{j=0}^{a_i-k}\mu_jy^j_\alpha y^{a_i-k-j}_\beta$
  for some $\mu_j\in \mathbb{K}$. So we have
  \begin{align*}
    g_k = & (\sum_{s=0}^{d-a_i}(-1)^s\binom{d-a_i}{s}y^s_\alpha y^{d-a_i-s}_\beta)(\sum_{j=0}^{a_i-k}\mu_jy^j_\alpha y^{a_i-k-j}_\beta) \\=&  \sum_{s=0}^{d-a_i}\sum_{j=0}^{a_i-k}(-1)^s\mu_j\binom{d-a_i}{s}y^{s+j}_{\alpha}y^{d-k-s-j}_{\beta}\\=&  \sum_{j=0}^{a_i-k}\sum_{l=j}^{j+d-a_i}(-1)^{l-j}\mu_j\binom{d-a_i}{l-j}y^l_\alpha y^{d-k-l}_\beta.
  \end{align*}
  We claim that $g_k$ has at most $a_i-k$ coefficients that are
  zero. Suppose $a_i-k+1$ coefficients in the above expression of
  $g_k$ are zero and consider the system of equations in the
  parameters $\mu_j$ corresponding to these coefficients being
  zero. Observe that the coefficient matrix of this system of
  equations is the transpose of a square submatrix of maximal rank of
  the Toeplitz matrix $T_{(a_i-k+1)\times (d-k+1)}$, up to
  multiplication of every second row and every second column by
  negative one. Using Lemma \ref{Toeplitz} we get that the determinant
  of this coefficient matrix is non-zero and this implies that all the
  parameters $\mu_j$ are zero hence $g_k$ is zero. Therefore for all
  $0\leq k\leq a_i$ the polynomial $g_k$ has at most
  $(a_i-k+1)-1=a_i-k$ zero terms. So we have
  $\vert {\mathcal{L}}^k_{i,d}\cap \supp(f) \vert=\vert
  \supp(g_k)\vert\geq (d-k+1)-(a_i-k)=d-a_i+1$
  for all $0\leq k\leq a_i$ and all $1\leq i\leq 3$.

\end{proof}
Now we are able to state and prove the main theorem of this section. Recall from  Definition~\ref{def:nu} that $\phi (I,d) : \times (x_1+x_1+x_3):(S/I)_{d-1} \rightarrow
  (S/I)_{d}$.
\begin{thm}\label{3d-3}
  For $d\geq 2$ we have that
  \[ \nu(3,d) = \left\{ \begin{array}{ll}
                          3d-3& \mbox{if $d$ is odd }\\
                          3d-2 & \mbox{if $d$ is even} .\end{array}
                                 \right. \]
                              Where, 
$
\nu(3,d) = min \lbrace H_{(S/I)}(d) \mid \phi(I,d) \hspace*{0.1
  cm}\text{is not surjective, for}\hspace*{1 mm}
I\in\mathcal{I}_d\rbrace,
$ and $\mathcal{I}_d$ is the set of all artinian monomial ideals of $S$ generated in degree $d$.
                               \end{thm}
                               \begin{proof}
                                First of all we observe that for $f=(y_1-y_2)(y_1-y_3)(y_2-y_3)^{d-2}$ we have $\ell\circ f=0$, where $\ell=x_1+x_2+x_3$ and since $\vert \supp(f)\vert = 3d-3$ for odd $d$ and $\vert \supp(f)\vert=3d-2$, we have $\nu(3,d)\leq 3d-3$ for odd $d$, and $\nu(3,d) \leq 3d-2$  for even $d$. \\
                                 To prove the equality,
                                 we check that for any $f\in (I^{-1})_d$ where, $\ell\circ f=0$, $\vert \supp(f)\vert \geq 3d-3$ for odd $d$ and $\vert \supp(f)\vert \geq 3d-2$ for even  $d$.\\We start by showing that $\vert \supp(f)\vert \geq 3d-3$ for all $d\geq 3$. Set $a_i=max \lbrace \deg_i(m) \mid m \in \supp(f) \rbrace $ for $1\leq i\leq 3$. Without loss of generality, we may assume that $a_1\leq a_2\leq a_3$. We can see $a_1\geq 2$. In fact by using Proposition \ref{level mons} we get that $\vert {\mathcal{L}}^0_{1,d}\cap \supp(f)\vert \geq d-{a_1}-1$. On the other hand since $I$ is an artinian ideal generated in degree $d$ we have $y^d_i\notin \supp(f)$ for each  $1\leq i\leq 3$ and this implies that $\vert {\mathcal{L}}^0_{1,d}\cap \supp(f)\vert \leq d-1$ and therefore, $a_1\geq 2$.\\
                                 Write $f=\sum_{j=0}^{a_1} y^j_1g_j$,
                                 where $g_j$ is a polynomial of degree
                                 $d-j$ in the variables $y_2$ and
                                 $y_3$. Using Proposition \ref{level
                                   mons} we get,
                                 $\vert {\mathcal{L}}^j_{1,d}\cap
                                 \supp(f) \vert \geq d-{a_1}+1$
                                 for all $0\leq j \leq
                                 a_1$. Therefore,
$$
\vert \supp(f) \vert \geq \sum_{j=0}^{a_1} \vert
\mathcal{L}^j_{1,d}\cap \supp(f) \vert \geq (a_1+1)(d-a_1+1).
$$ 
So
$\vert \supp(f) \vert\geq (a_1+1)(d-a_1+1) = 3(d-1) + (a_1-2)(d-2-a_1)
\geq 3d-3$,
for $ 2\leq a_1\leq d-2$.  Furthermore, strict inequality holds for
$ 2<a_1< d-2$, which means $\vert\supp(f)\vert\geq 3d-2$, for all
$ 2<a_1< d-2$. It remains to consider the cases where $a_1=2$ and
$a_1\geq d-2$.  \\If $a_1=2$, ideal
$J = (x^3_1, x^{d}_2,x^{d}_3)\subset S$ is an artinian monomial
complete intersection and by Theorem \ref{CI}, $J$ has strong
Lefschetz property. The Hilbert series of $S/J$ shows that there is a
unique generator for the kernel of the differentiation map
$\circ (x_1+x_2+x_3):(J^{-1})_{d}\longrightarrow (J^{-1})_{d}$. Since the
polynomial $(y_1-y_2)(y_1-y_3)(y_2-y_3)^{d-2}$ is in the kernel of
this map and has $3d-2$ non-zero terms for even degree $d$ we have
$\vert \supp(f)\vert \geq 3d-2$ for any homogeneous degree $d$ form
$f$ where $\ell\circ f=0$.

Now suppose that $a_1\geq d-2$, then since $a_1\leq a_2\leq a_3\leq d-1$, all possible choices for the triple $(a_1,a_2,a_3)$ are  $(d-1,d-1,d-1)$, $(d-2, d-1, d-1),(d-2, d-2, d-1)$ and $(d-2, d-2, d-2)$. \\
First, we consider the case $(a_1,a_2,a_3)=(d-1,d-1,d-1)$. Proposition
\ref{level mons} implies that
$\vert \mathcal{L}_{i,d}^k\cap \supp(f)\vert \geq d-(d-1)+1=2$ for all
$0\leq k\leq d-1$ and $1\leq i\leq 3$.  \\ If $d$ is odd, we consider
$\supp(f)= \sqcup_{i=1}^3 (\cup_{(d+1)/2}^{d-1}
\mathcal{L}^k_{i,d}\cap \supp(f)) $
as a partition for $\supp(f)$. Therefore,
\begin{align*}
  \vert \supp(f)\vert = & \vert  \cup_{(d+1)/2}^{d-1} \mathcal{L}^k_{1,d}\cap \supp(f)\vert +\vert \cup_{(d+1)/2}^{d-1} \mathcal{L}^k_{2,d}\cap \supp(f)\vert + \vert \cup_{(d+1)/2}^{d-1} \mathcal{L}^k_{3,d}\cap \supp(f)\vert \\   \geq & 3\times 2\times (d-1-(d+1)/2+1) =3d-3.
\end{align*}
If $d$ is even, we consider, $\supp(f)=(\sqcup_{i=1}^3\mathcal{A}_i)\sqcup \mathcal{B}\sqcup\mathcal{C}$ be a partition for $\supp(f)$ where $\vert \mathcal{A}_i\cap \mathcal{L}^k_{i,d}\vert =2$, for each $1\leq i\leq 3$ and ${(d-2)}/2\leq k\leq d-1$ and since each pair of the sets $\mathcal{L}^{d/2}_{i,d}$ for $1\leq i\leq 3$ has intersection in two monomials we get $\vert(\cup_{i=1}^3 \mathcal{L}^{d/2}_{i,d})\cap \supp(f)\vert\geq 3\times 2 - 3=3$, so we can choose $\vert \mathcal{B}\cap (\cup_{i=1}^3 \mathcal{L}^{d/2}_{i,d})\vert=3$. Note that $\vert (\sqcup_{i=1}^3\mathcal{A}_i)\sqcup \mathcal{B}\vert = 3\times 2\times {(d-2)}/2 +3 = 3d-3$. If $\vert \mathcal{C}\vert=0$, the set  $\mathcal{B}\cap (\cup_{i=1}^3 \mathcal{L}^{d/2}_{i,d})$ contains exactly three monomials in the pairwise intersection of $\mathcal{L}^{d/2}_{i,d}$ for $1\leq i\leq 3$. Using Proposition \ref{prop:levels} part $(i)$, the sum of the coefficients of $f$ corresponding to the monomials in $\mathcal{L}^{d/2}_{i,d}\cap \supp(f)$ is zero for each $1\leq i\leq 3$, which implies  the sum of the coefficients of each pair of the monomials in  $\mathcal{B}\cap (\cup_{i=1}^3 \mathcal{L}^{d/2}_{i,d})$ is zero and this means all of them have to be zero. So $\vert \mathcal{C}\vert \geq 1$, so $\vert \supp(f)\vert \geq 3d-2$.(see Figure \ref{fig1})\\

\definecolor{uuuuuu}{rgb}{0.26666666666666666,0.26666666666666666,0.26666666666666666}
\definecolor{uququq}{rgb}{0.25098039215686274,0.25098039215686274,0.25098039215686274}
\begin{figure}
  \begin{tikzpicture}
    [line cap=round,line join=round,>=triangle 45,x=0.6cm,y=0.6cm]
    \clip(-1.,-2.15) rectangle (11.,10.); \fill[line
    width=1.2pt,color=uququq,fill=uququq,fill
    opacity=0.10000000149011612]
    (0.4896504848090458,0.8480995176399997) --
    (1.98464150588864,3.4374999230091303) -- (4.,0.) --
    (0.9793009696180911,0.) -- cycle; \fill[line
    width=1.2pt,color=uququq,fill=uququq,fill
    opacity=0.10000000149011612] (4.492032596627272,7.780428686613983)
    -- (5.507967403372732,7.780428686613983) --
    (6.987582851133786,5.217659555428066) --
    (3.012417148866218,5.217659555428066) -- cycle; \fill[line
    width=1.2pt,color=uququq,fill=uququq,fill
    opacity=0.10000000149011612]
    (8.011836611624666,3.4436000024143736) --
    (9.49417148302896,0.8761206913110575) -- (9.,0.) --
    (6.0236732232493315,0.) -- cycle; \draw [line width=1.2pt]
    (9.,0.)-- (9.49417148302896,0.8761206913110586); \draw [line
    width=1.2pt] (8.007327664869802,0.)--
    (9.003663832434903,1.7257048636412096); \draw [line width=1.2pt]
    (7.023776169101887,0.)-- (8.511888084550945,2.577485444906408);
    \draw [line width=1.2pt] (6.0236732232493315,0.)--
    (8.011836611624666,3.4436000024143745); \draw [line width=2.pt]
    (5.040575355193507,-0.009233209950459086)--
    (7.522953075721945,4.290371125581811); \draw [line width=1.2pt]
    (0.9793009696180911,0.)-- (0.4896504848090458,0.8480995176399997);
    \draw [line width=1.2pt] (2.001338878616366,0.)--
    (1.0006694393081834,1.7332103104632341); \draw [line width=1.2pt]
    (2.993613526177205,0.)-- (1.4968067630886037,2.5925453627821717);
    \draw [line width=1.2pt] (3.969283011777278,0.)--
    (1.98464150588864,3.4374999230091303); \draw [line width=2.pt]
    (5.040575355193507,-0.009233209950459086)--
    (2.517622279471564,4.360649702312117); \draw [line width=1.2pt]
    (4.492032596627272,7.780428686613984)--
    (5.507967403372733,7.780428686613984); \draw [line width=1.2pt]
    (3.991880895975648,6.914140529593392)--
    (6.010399308523301,6.910191099549572); \draw [line width=1.2pt]
    (3.5125688495178418,6.083947712448658)--
    (6.487431150482164,6.083947712448658); \draw [line width=1.2pt]
    (3.0124171488662173,5.217659555428064)--
    (6.987582851133785,5.217659555428064); \draw [line width=2.pt]
    (2.517622279471564,4.360649702312117)--
    (7.522953075721945,4.290371125581811); \draw [line
    width=1.2pt,color=uququq]
    (0.4896504848090458,0.8480995176399997)--
    (1.98464150588864,3.4374999230091303); \draw [line
    width=1.2pt,color=uququq] (1.98464150588864,3.4374999230091303)--
    (4.,0.); \draw [line width=1.2pt,color=uququq] (4.,0.)--
    (0.9793009696180911,0.); \draw [line width=1.2pt,color=uququq]
    (0.9793009696180911,0.)-- (0.4896504848090458,0.8480995176399997);
    \draw [line width=1.2pt,color=uququq]
    (4.492032596627272,7.780428686613983)--
    (5.507967403372732,7.780428686613983); \draw [line
    width=1.2pt,color=uququq] (5.507967403372732,7.780428686613983)--
    (6.987582851133786,5.217659555428066); \draw [line
    width=1.2pt,color=uququq] (6.987582851133786,5.217659555428066)--
    (3.012417148866218,5.217659555428066); \draw [line
    width=1.2pt,color=uququq] (3.012417148866218,5.217659555428066)--
    (4.492032596627272,7.780428686613983); \draw [line
    width=1.2pt,color=uququq] (8.011836611624666,3.4436000024143736)--
    (9.49417148302896,0.8761206913110575); \draw [line
    width=1.2pt,color=uququq] (9.49417148302896,0.8761206913110575)--
    (9.,0.); \draw [line width=1.2pt,color=uququq] (9.,0.)--
    (6.0236732232493315,0.); \draw [line width=1.2pt,color=uququq]
    (6.0236732232493315,0.)-- (8.011836611624666,3.4436000024143736);
    \draw (1.1225001355992835,1.8085874015531886) node[anchor=north
    west] {$\mathcal{A}_2$}; \draw
    (4.1187098108281765,7.037175865882396) node[anchor=north west]
    {{${\mathcal{A}_1}$}}; \draw
    (7.586684987260156,1.8444701521547322) node[anchor=north west]
    {$\mathcal{A}_3$}; \draw (4.644299437933579,9.738675284494299)
    node[anchor=north west] {$y_1^d$}; \draw
    (9.990829277563579,0.7679876341084276) node[anchor=north west]
    {$y_3^d$}; \draw (-1.0457614041025956,0.7500462588076559)
    node[anchor=north west] {$y_2^d$}; \draw (3.644299437933579,-0.5)
    node[anchor=north west] {$d$ is even}; \draw [line width=1.2pt]
    (0.,0.)-- (5.,8.660254037844387); \draw [line width=1.2pt]
    (5.,8.660254037844387)-- (10.,0.); \draw [line width=1.2pt]
    (10.,0.)-- (0.,0.); \begin{scriptsize} \draw [fill=black]
      (2.517622279471564,4.360649702312118) circle (2.5pt); \draw
      [fill=black] (7.522953075721946,4.29037112558181) circle
      (2.5pt); \draw [fill=black] (5.04590615144389,0.) circle
      (2.5pt); \draw [fill=uuuuuu] (0.,0.) circle (2.0pt); \draw
      [fill=uuuuuu] (5.,8.660254037844387) circle (2.0pt); \draw
      [fill=uuuuuu] (10.,0.) circle (2.0pt); \end{scriptsize}
  \end{tikzpicture}
  \begin{tikzpicture}[line cap=round,line join=round,>=triangle
    45,x=0.66cm,y=0.66cm]
    \clip(-1.,-2.) rectangle (10.,9.); \fill[line
    width=1.2pt,color=uququq,fill=uququq,fill
    opacity=0.10000000149011612] (3.997275263841643,6.923483848812012)
    -- (5.017949192431124,6.897114317029975) --
    (6.018393472388334,5.164293994002423) --
    (3.008715016258719,5.211247273655521) -- cycle; \fill[line
    width=1.2pt,color=uququq,fill=uququq,fill
    opacity=0.10000000149011612]
    (0.502666688159133,0.8706442431639989) --
    (1.5066964479989235,2.6096747995176934) -- (3.018421220600845,0.)
    -- (1.0022306950835727,0.) -- cycle; \fill[line
    width=1.2pt,color=uququq,fill=uququq,fill
    opacity=0.10000000149011612] (7.523478768496922,2.5574097914175)
    -- (8.517209438086388,0.8362177826491061) --
    (8.030408142892549,0.) -- (6.053742279355587,0.) -- cycle; \draw
    [line width=1.2pt] (0.,0.)-- (4.5,7.794228634059948); \draw [line
    width=1.2pt] (4.5,7.794228634059948)-- (9.,0.); \draw [line
    width=1.2pt] (9.,0.)-- (0.,0.); \draw [line width=1.2pt]
    (3.997275263841643,6.923483848812012)--
    (5.017949192431124,6.897114317029975); \draw [line width=1.2pt]
    (3.5059085006822355,6.072411649869255)--
    (5.510517614366132,6.043960784034516); \draw [line width=1.2pt]
    (3.008715016258719,5.211247273655521)--
    (6.018393472388334,5.164293994002423); \draw [line width=1.2pt]
    (2.5167258979619582,4.359097123994515)--
    (6.5162395692966,4.301999259807449); \draw [line width=1.2pt]
    (4.036965511460288,0.00777062574749973)--
    (2.0207259421636903,3.5); \draw [line width=1.2pt]
    (7.022953179664346,3.4243455417638518)--
    (5.039679077626302,-0.010785968301580214); \draw [line
    width=1.2pt,color=uququq] (3.997275263841643,6.923483848812012)--
    (5.017949192431124,6.897114317029975); \draw [line
    width=1.2pt,color=uququq] (5.017949192431124,6.897114317029975)--
    (6.018393472388334,5.164293994002423); \draw [line
    width=1.2pt,color=uququq] (6.018393472388334,5.164293994002423)--
    (3.008715016258719,5.211247273655521); \draw [line
    width=1.2pt,color=uququq] (3.008715016258719,5.211247273655521)--
    (3.997275263841643,6.923483848812012); \draw [line width=1.2pt]
    (1.5066964479989235,2.6096747995176934)-- (3.018421220600845,0.);
    \draw [line width=1.2pt] (1.0145976718395366,1.7573347168671714)--
    (2.0216585484887846,0.); \draw [line width=1.2pt]
    (0.502666688159133,0.8706442431639989)-- (1.0022306950835727,0.);
    \draw [line width=1.2pt] (7.523478768496922,2.5574097914175)--
    (6.053742279355587,0.); \draw [line width=1.2pt]
    (8.021324865405187,1.6951150572225253)-- (7.046922501193663,0.);
    \draw [line width=1.2pt] (8.517209438086388,0.8362177826491061)--
    (8.030408142892549,0.); \draw [line width=1.2pt,color=uququq]
    (0.502666688159133,0.8706442431639989)--
    (1.5066964479989235,2.6096747995176934); \draw [line
    width=1.2pt,color=uququq]
    (1.5066964479989235,2.6096747995176934)-- (3.018421220600845,0.);
    \draw [line width=1.2pt,color=uququq] (3.018421220600845,0.)--
    (1.0022306950835727,0.); \draw [line width=1.2pt,color=uququq]
    (1.0022306950835727,0.)-- (0.502666688159133,0.8706442431639989);
    \draw [line width=1.2pt,color=uququq]
    (7.523478768496922,2.5574097914175)--
    (8.517209438086388,0.8362177826491061); \draw [line
    width=1.2pt,color=uququq] (8.517209438086388,0.8362177826491061)--
    (8.030408142892549,0.); \draw [line width=1.2pt,color=uququq]
    (8.030408142892549,0.)-- (6.053742279355587,0.); \draw [line
    width=1.2pt,color=uququq] (6.053742279355587,0.)--
    (7.523478768496922,2.5574097914175); \draw
    (4.165695921154608,8.737586171670088) node[anchor=north west]
    {$y_1^d$}; \draw (9.032251606344808,0.7046497041286413)
    node[anchor=north west] {$y_3^d$}; \draw
    (-1.0724370478083225,0.7514435087939313) node[anchor=north west]
    {$y_2^d$}; \draw (3.165695921154608,-0.5) node[anchor=north west]
    {$d$ is odd};
    \begin{scriptsize}
      \draw [fill=uuuuuu] (4.5,7.794228634059948) circle (1.5pt);
      \draw [fill=uuuuuu] (9.,0.) circle (1.5pt); \draw [fill=uuuuuu]
      (0.,0.) circle (1.5pt); \draw [fill=uuuuuu]
      (3.997275263841643,6.923483848812012) circle (1.5pt); \draw
      [fill=uuuuuu] (3.5059085006822355,6.072411649869255) circle
      (1.5pt); \draw [fill=uuuuuu]
      (3.008715016258719,5.211247273655521) circle (1.5pt); \draw
      [fill=uuuuuu] (2.5167258979619582,4.359097123994515) circle
      (1.5pt); \draw [fill=uuuuuu] (2.0207259421636903,3.5) circle
      (1.5pt); \draw [fill=uuuuuu]
      (1.5066964479989235,2.6096747995176934) circle (1.5pt); \draw
      [fill=uuuuuu] (1.0145976718395366,1.7573347168671714) circle
      (1.5pt); \draw [fill=uuuuuu]
      (0.502666688159133,0.8706442431639989) circle (1.5pt); \draw
      [fill=uuuuuu] (4.036965511460288,0.00777062574749973) circle
      (1.5pt); \draw [fill=uuuuuu]
      (5.039679077626302,-0.010785968301580214) circle (1.5pt); \draw
      [fill=uuuuuu] (8.517209438086388,0.8362177826491061) circle
      (1.5pt); \draw [fill=uuuuuu]
      (8.021324865405187,1.6951150572225253) circle (1.5pt); \draw
      [fill=uuuuuu] (7.523478768496922,2.5574097914175) circle
      (1.5pt); \draw [fill=uuuuuu]
      (7.022953179664346,3.4243455417638518) circle (1.5pt); \draw
      [fill=uuuuuu] (6.5162395692966,4.301999259807449) circle
      (1.5pt); \draw [fill=uuuuuu]
      (6.018393472388334,5.164293994002423) circle (1.5pt); \draw
      [fill=uuuuuu] (5.510517614366132,6.043960784034516) circle
      (1.5pt); \draw [fill=uuuuuu]
      (5.017949192431124,6.897114317029975) circle (1.5pt); \draw
      [fill=uuuuuu] (6.053742279355587,0.) circle (1.5pt); \draw
      [fill=uuuuuu] (7.046922501193663,0.) circle (1.5pt); \draw
      [fill=uuuuuu] (8.030408142892549,0.) circle (1.5pt); \draw
      [fill=uuuuuu] (3.018421220600845,0.) circle (1.5pt); \draw
      [fill=uuuuuu] (2.0216585484887846,0.) circle (1.5pt); \draw
      [fill=uuuuuu] (1.0022306950835727,0.) circle (1.5pt);
    \end{scriptsize}
  \end{tikzpicture}
  \caption{ For even $d$, the middle downward triangle is
    $\mathcal{B}$.}\label{fig1}
\end{figure}
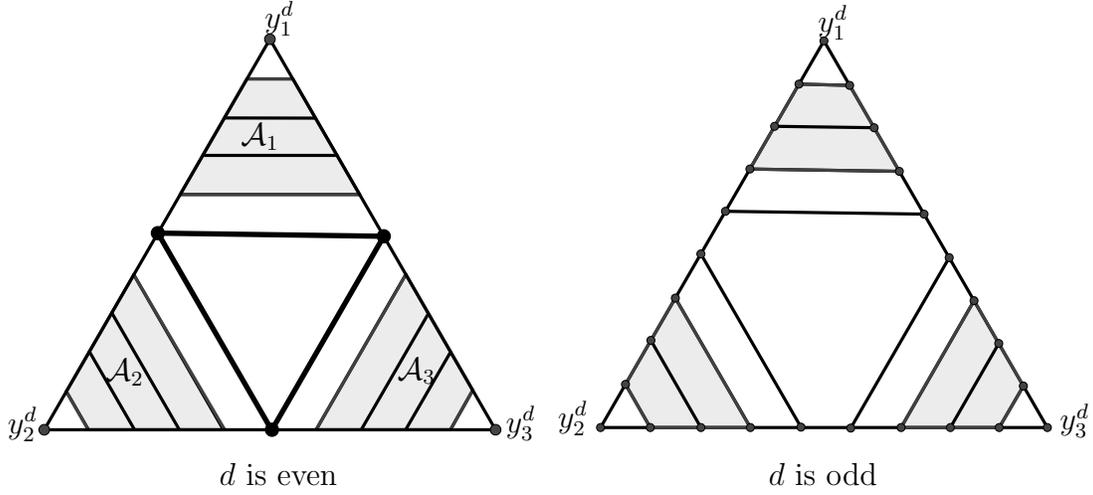
For the three remaining cases where $a_1=d-2$ we will show for even
degree $d$ we get $\vert \supp (f)\vert \geq 3d-2$, see Figure
\ref{fig2}. Note that when $d=4$ and $a_1=d-2=2$ we have seen already
that $\vert \supp(f)\vert \geq 3d-2$. So we can assume $d\geq 6$.
Using Proposition \ref{level mons} we get
$\vert \mathcal{L}^j_{1,d}\cap \supp(f)\vert \geq d-(d-2)+1 = 3$ for
each $0\leq j\leq d-2$, so we can partition
$ \supp(f) = \mathcal{S}_1\sqcup \mathcal{S}_2$, where
$\vert \mathcal{S}_1 \cap \mathcal{L}^j_{1,d}\vert=3$ for
$3\leq j\leq d-2$ .Then assume $d\geq 6$ and consider the following
cases.

If $(a_1,a_2,a_3)= (d-2, d-2, d-2)$ we apply Proposition \ref{level
  mons} for the variables $y_2$ and $y_3$, where $d-3\leq j,k\leq d-2$
\begin{align*}
  \vert \supp(f)\vert & = \vert \mathcal{S}_1 \vert + \mid \left(\cup_{j=d-3}^{d-2}(\mathcal{L}^j_{2,d}\cap \supp(f))\right) \cup \left(\cup_{k=d-3}^{d-2}(\mathcal{L}^k_{3,d}\cap \supp(f))\right)\setminus \left(\mathcal{S}_1\cap \mathcal{L}^3_{1,d}\right) \mid \\ &\geq 3(d-4)+ (2\times 3-1)+(2\times 3-1)= 3d-2.
\end{align*}
If $(a_1,a_2,a_3)=(d-2, d-2, d-1)$, we apply Proposition \ref{level
  mons} for the variables $y_2$ and $y_3$, where $d-3\leq j\leq d-2$
and $d-3\leq k\leq d-1$
\begin{align*}
  \vert \supp(f)\vert & = \vert \mathcal{S}_1 \vert +\mid \left(\cup_{j=d-3}^{d-2}(\mathcal{L}^j_{2,d}\cap \supp(f))\right)\cup \left(\cup_{k=d-3}^{d-1}(\mathcal{L}^k_{3,d}\cap \supp(f))\right)\setminus \left(\mathcal{S}_1\cap \mathcal{L}^3_{1,d}\right)\mid \\ &\geq 3(d-4)+ (2\times 3-1)+(3\times 2-1)= 3d-2.
\end{align*}
If $(a_1,a_2,a_3)=(d-2, d-1, d-1)$, we apply Proposition \ref{level
  mons} for the variables $y_2$ and $y_3$, where $d-3\leq j,k\leq d-1$
\begin{align*}
  \vert \supp(f)\vert & = \vert \mathcal{S}_1 \vert +\mid \left(\cup_{j=d-3}^{d-1}(\mathcal{L}^j_{2,d}\cap \supp(f))\right)\cup \left(\cup_{k=d-3}^{d-1}(\mathcal{L}^k_{3,d}\cap \supp(f))\right)\setminus \left(\mathcal{S}_1\cap \mathcal{L}^3_{1,d}\right) \mid \\ &\geq 3(d-4)+ (3\times 2-1)+(3\times 2-1)= 3d-2.
\end{align*}
\definecolor{uququq}{rgb}{0.25098039215686274,0.25098039215686274,0.25098039215686274}
\definecolor{uuuuuu}{rgb}{0.26666666666666666,0.26666666666666666,0.26666666666666666}
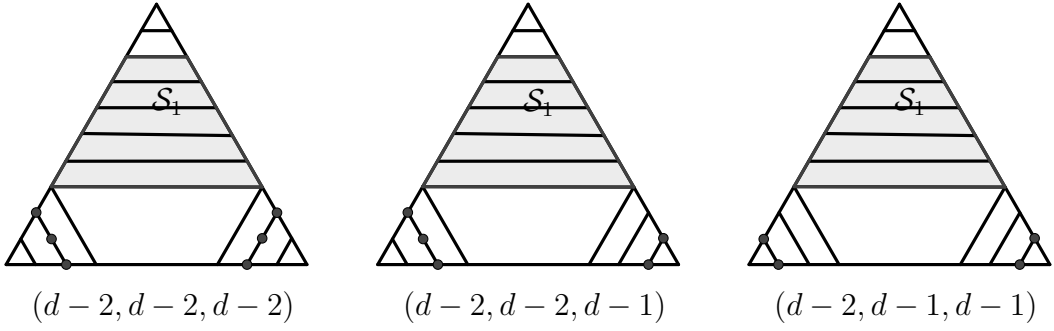
\begin{figure}
  \begin{tikzpicture}[line cap=round,line join=round,>=triangle
    45,x=0.4cm,y=0.4cm]
    \clip(-1.,-2.) rectangle (11.,10.); \fill[line
    width=1.2pt,color=uququq,fill=uququq,fill
    opacity=0.10000000149011612] (3.991880895975647,6.91414052959339)
    -- (6.010399308523302,6.910191099549571) --
    (8.511888084550945,2.5774854449064084) --
    (1.4968067630886037,2.592545362782172) -- cycle; \draw [line
    width=1.2pt] (9.,0.)-- (9.49417148302896,0.8761206913110586);
    \draw [line width=1.2pt] (8.007327664869802,0.)--
    (9.003663832434903,1.7257048636412096); \draw [line width=1.2pt]
    (7.023776169101887,0.)-- (8.511888084550945,2.577485444906408);
    \draw [line width=1.2pt] (0.9793009696180911,0.)--
    (0.4896504848090458,0.8480995176399997); \draw [line width=1.2pt]
    (2.001338878616366,0.)-- (1.0006694393081834,1.7332103104632341);
    \draw [line width=1.2pt] (2.993613526177205,0.)--
    (1.4968067630886037,2.5925453627821717); \draw [line width=1.2pt]
    (4.492032596627272,7.780428686613984)--
    (5.507967403372733,7.780428686613984); \draw [line width=1.2pt]
    (3.991880895975648,6.914140529593392)--
    (6.010399308523301,6.910191099549572); \draw [line width=1.2pt]
    (3.5125688495178418,6.083947712448658)--
    (6.487431150482164,6.083947712448658); \draw [line width=1.2pt]
    (3.0124171488662173,5.217659555428064)--
    (6.987582851133785,5.217659555428064); \draw [line width=1.2pt]
    (2.517622279471564,4.360649702312117)--
    (7.522953075721945,4.290371125581811); \draw [line width=1.2pt]
    (0.,0.)-- (5.,8.660254037844387); \draw [line width=1.2pt]
    (5.,8.660254037844387)-- (10.,0.); \draw [line width=1.2pt]
    (10.,0.)-- (0.,0.); \draw [line width=1.2pt]
    (1.98464150588864,3.4374999230091303)--
    (8.011836611624666,3.4436000024143736); \draw [line width=1.2pt]
    (1.4968067630886037,2.592545362782172)--
    (8.511888084550945,2.5774854449064084); \draw [line
    width=1.2pt,color=uququq] (3.991880895975647,6.91414052959339)--
    (6.010399308523302,6.910191099549571); \draw [line
    width=1.2pt,color=uququq] (6.010399308523302,6.910191099549571)--
    (8.511888084550945,2.5774854449064084); \draw [line
    width=1.2pt,color=uququq] (8.511888084550945,2.5774854449064084)--
    (1.4968067630886037,2.592545362782172); \draw [line
    width=1.2pt,color=uququq] (1.4968067630886037,2.592545362782172)--
    (3.991880895975647,6.91414052959339); \draw
    (4.479828891009773,6.213543581827591) node[anchor=north west]
    {$\mathcal{S}_1$}; \draw (0.479828891009773,-0.5213543581827591)
    node[anchor=north west] {$(d-2,d-2,d-2)$};
    \begin{scriptsize}
      \draw [fill=uuuuuu] (9.003663832434903,1.725704863641209) circle
      (1.8pt); \draw [fill=uuuuuu]
      (8.511390461823932,0.8730623745298263) circle (1.8pt); \draw
      [fill=uuuuuu] (8.007327664869802,0.) circle (1.8pt); \draw
      [fill=uuuuuu] (1.0006694393081834,1.733210310463234) circle
      (1.8pt); \draw [fill=uuuuuu]
      (1.5098554358395553,0.8512742939683074) circle (1.8pt); \draw
      [fill=uuuuuu] (2.001338878616366,0.) circle (1.8pt);
    \end{scriptsize}
  \end{tikzpicture}
  \begin{tikzpicture}[line cap=round,line join=round,>=triangle
    45,x=0.4cm,y=0.4cm]
    \clip(-1.,-2.) rectangle (11.,10.); \fill[line
    width=1.2pt,color=uququq,fill=uququq,fill
    opacity=0.10000000149011612] (3.991880895975647,6.91414052959339)
    -- (6.010399308523302,6.910191099549571) --
    (8.511888084550945,2.5774854449064084) --
    (1.4968067630886037,2.592545362782172) -- cycle; \draw [line
    width=1.2pt] (9.,0.)-- (9.49417148302896,0.8761206913110586);
    \draw [line width=1.2pt] (8.007327664869802,0.)--
    (9.003663832434903,1.7257048636412096); \draw [line width=1.2pt]
    (7.023776169101887,0.)-- (8.511888084550945,2.577485444906408);
    \draw [line width=1.2pt] (0.9793009696180911,0.)--
    (0.4896504848090458,0.8480995176399997); \draw [line width=1.2pt]
    (2.001338878616366,0.)-- (1.0006694393081834,1.7332103104632341);
    \draw [line width=1.2pt] (2.993613526177205,0.)--
    (1.4968067630886037,2.5925453627821717); \draw [line width=1.2pt]
    (4.492032596627272,7.780428686613984)--
    (5.507967403372733,7.780428686613984); \draw [line width=1.2pt]
    (3.991880895975648,6.914140529593392)--
    (6.010399308523301,6.910191099549572); \draw [line width=1.2pt]
    (3.5125688495178418,6.083947712448658)--
    (6.487431150482164,6.083947712448658); \draw [line width=1.2pt]
    (3.0124171488662173,5.217659555428064)--
    (6.987582851133785,5.217659555428064); \draw [line width=1.2pt]
    (2.517622279471564,4.360649702312117)--
    (7.522953075721945,4.290371125581811); \draw [line width=1.2pt]
    (0.,0.)-- (5.,8.660254037844387); \draw [line width=1.2pt]
    (5.,8.660254037844387)-- (10.,0.); \draw [line width=1.2pt]
    (10.,0.)-- (0.,0.); \draw [line width=1.2pt]
    (1.98464150588864,3.4374999230091303)--
    (8.011836611624666,3.4436000024143736); \draw [line width=1.2pt]
    (1.4968067630886037,2.592545362782172)--
    (8.511888084550945,2.5774854449064084); \draw [line
    width=1.2pt,color=uququq] (3.991880895975647,6.91414052959339)--
    (6.010399308523302,6.910191099549571); \draw [line
    width=1.2pt,color=uququq] (6.010399308523302,6.910191099549571)--
    (8.511888084550945,2.5774854449064084); \draw [line
    width=1.2pt,color=uququq] (8.511888084550945,2.5774854449064084)--
    (1.4968067630886037,2.592545362782172); \draw [line
    width=1.2pt,color=uququq] (1.4968067630886037,2.592545362782172)--
    (3.991880895975647,6.91414052959339); \draw
    (4.471136546982098,6.195046384358876) node[anchor=north west]
    {$\mathcal{S}_1$}; \draw (0.479828891009773,-0.5213543581827591)
    node[anchor=north west] {$(d-2,d-2,d-1)$};
    \begin{scriptsize}
      \draw [fill=uuuuuu] (9.49417148302896,0.8761206913110575) circle
      (1.8pt); \draw [fill=uuuuuu] (9.,0.) circle (1.8pt); \draw
      [fill=uuuuuu] (1.5098554358395553,0.8512742939683074) circle
      (1.8pt); \draw [fill=uuuuuu]
      (1.0006694393081834,1.733210310463234) circle (1.8pt); \draw
      [fill=uuuuuu] (2.001338878616366,0.) circle (1.8pt);
    \end{scriptsize}
  \end{tikzpicture}
  \begin{tikzpicture}[line cap=round,line join=round,>=triangle
    45,x=0.4cm,y=0.4cm]
    \clip(-1.,-2.) rectangle (11.,10.); \fill[line
    width=1.2pt,color=uququq,fill=uququq,fill
    opacity=0.10000000149011612] (3.991880895975647,6.91414052959339)
    -- (6.010399308523302,6.910191099549571) --
    (8.511888084550945,2.5774854449064084) --
    (1.4968067630886037,2.592545362782172) -- cycle; \draw [line
    width=1.2pt] (9.,0.)-- (9.49417148302896,0.8761206913110586);
    \draw [line width=1.2pt] (8.007327664869802,0.)--
    (9.003663832434903,1.7257048636412096); \draw [line width=1.2pt]
    (7.023776169101887,0.)-- (8.511888084550945,2.577485444906408);
    \draw [line width=1.2pt] (0.9793009696180911,0.)--
    (0.4896504848090458,0.8480995176399997); \draw [line width=1.2pt]
    (2.001338878616366,0.)-- (1.0006694393081834,1.7332103104632341);
    \draw [line width=1.2pt] (2.993613526177205,0.)--
    (1.4968067630886037,2.5925453627821717); \draw [line width=1.2pt]
    (4.492032596627272,7.780428686613984)--
    (5.507967403372733,7.780428686613984); \draw [line width=1.2pt]
    (3.991880895975648,6.914140529593392)--
    (6.010399308523301,6.910191099549572); \draw [line width=1.2pt]
    (3.5125688495178418,6.083947712448658)--
    (6.487431150482164,6.083947712448658); \draw [line width=1.2pt]
    (3.0124171488662173,5.217659555428064)--
    (6.987582851133785,5.217659555428064); \draw [line width=1.2pt]
    (2.517622279471564,4.360649702312117)--
    (7.522953075721945,4.290371125581811); \draw [line width=1.2pt]
    (0.,0.)-- (5.,8.660254037844387); \draw [line width=1.2pt]
    (5.,8.660254037844387)-- (10.,0.); \draw [line width=1.2pt]
    (10.,0.)-- (0.,0.); \draw [line width=1.2pt]
    (1.98464150588864,3.4374999230091303)--
    (8.011836611624666,3.4436000024143736); \draw [line width=1.2pt]
    (1.4968067630886037,2.592545362782172)--
    (8.511888084550945,2.5774854449064084); \draw [line
    width=1.2pt,color=uququq] (3.991880895975647,6.91414052959339)--
    (6.010399308523302,6.910191099549571); \draw [line
    width=1.2pt,color=uququq] (6.010399308523302,6.910191099549571)--
    (8.511888084550945,2.5774854449064084); \draw [line
    width=1.2pt,color=uququq] (8.511888084550945,2.5774854449064084)--
    (1.4968067630886037,2.592545362782172); \draw [line
    width=1.2pt,color=uququq] (1.4968067630886037,2.592545362782172)--
    (3.991880895975647,6.91414052959339); \draw
    (4.479828891009773,6.213543581827591) node[anchor=north west]
    {$\mathcal{S}_1$}; \draw (0.479828891009773,-0.5213543581827591)
    node[anchor=north west] {$(d-2,d-1,d-1)$};
    \begin{scriptsize}
      \draw [fill=uuuuuu] (9.49417148302896,0.8761206913110575) circle
      (1.8pt); \draw [fill=uuuuuu] (9.,0.) circle (1.8pt); \draw
      [fill=uuuuuu] (0.9793009696180911,0.) circle (1.8pt); \draw
      [fill=uuuuuu] (0.4896504848090458,0.8480995176399997) circle
      (1.8pt);
    \end{scriptsize}
  \end{tikzpicture}
  \caption{Three cases for the triple $(a_1,a_2,a_3)$ when $d$ is
    even.}\label{fig2}
\end{figure}
\end{proof}

\section{Bound on the number of generators of ideals with more than three variables failing WLP }
In this section we consider artinian monomial ideals
$I\subset S=\mathbb{K}[x_1,\dots ,x_n]$ generated in degree $d$, for
$n\geq 4$. We provide a sharp lower bound for the number of monomials
with non-zero coefficients in a non-zero form $f\in(I^{-1})_d$ such
that $(x_1+\cdots +x_n)\circ f=0$. The next theorem provides such
lower bound for the form $f$ in terms of the maximum degree of the
variables in $f$.
\begin{thm}\label{max}
  For $n\geq 4$ and $d\geq 2$, let $f$ be a non-zero form of degree
  $d$ in the dual ring $R=\mathbb{K}[y_1,\dots ,y_n]$ of
  $S=\mathbb{K}[x_1,\dots ,x_n]$ such that
  $(x_1+\cdots +x_n)\circ f =0$. Then we have
  $\vert \supp(f)\vert \geq \max\lbrace (a_i+1)(d-a_i+1)\mid a_i\neq
  0\rbrace$,
  where $a_i = \max \lbrace \deg_i(m)\mid m\in \supp(f)\rbrace$.
\end{thm}
\begin{proof}
  We show that for each $1\leq i\leq n$ we have
  $\vert \supp(f)\vert \geq (a_i+1)(d-a_i+1)$. Denote
  $\ell = x_1+\cdots +x_n$ and $\ell^\prime = \ell-x_i$ and write
  $f=\sum^{a_i}_{k=0}y^k_ig_k$, where $g_k$ is a polynomial of degree
  $d-k$ in the variables different from $y_i$. Since we have
  $\ell\circ f=0$, Lemma \ref{sur} implies that
$$
(k+1)g_{k+1}+\ell^\prime\circ g_k=0, \hspace*{0.5 cm}\forall 0\leq
k\leq a_i
$$
and acting on each equation by $(\ell^\prime)^{a_i-k}$ we get that
$(\ell^\prime)^{a_n-k+1}\circ g_k=0$ for all $0\leq k\leq a_i$. By the
definition of $a_i$ we have $g_{a_i}\neq 0$, Proposition
\ref{prop:levels} part (ii) implies that for every $0\leq k\leq a_i$
we have $g_k\neq 0$. Now applying Theorem \ref{a+2} we get that for
each $0\leq k\leq a_i$,
$\vert \supp(g_k)\vert \geq d-k-(a_i-k+1)+2=d-a_i+1$. Therefore,
\begin{align*}
  \vert \supp(f)\vert = &\vert \cup_{k=0}^{a_i}  \mathcal{L}_{i,d}^{k}\cap \supp(f)\vert = \sum^{a_i}_{k=0}\vert \supp(g_k)\vert\geq  (a_i+1)(d-a_i+1)
\end{align*}
and we conclude that
$\vert \supp(f)\vert \geq \max\lbrace (a_i+1)(d-a_i+1)\mid 1\leq i\leq
n\rbrace$.
\end{proof}
In general we can prove that the sharp lower bound is always $2d$. Recall from  Definition~\ref{def:nu} that $\phi (I,d) : \times (x_1+\cdots +x_n):(S/I)_{d-1} \rightarrow
  (S/I)_{d}$.
\begin{thm}\label{2d}
  For $n\geq 4$ and $d\geq 2$, we have
$$
\nu(n,d)=2d.
$$
 Where, 
$
\nu(n,d) = min \lbrace H_{(S/I)}(d) \mid \phi(I,d) \hspace*{0.1
  cm}\text{is not surjective, for}\hspace*{1 mm}
I\in\mathcal{I}_d\rbrace,
$ and $\mathcal{I}_d$ is the set of all artinian monomial ideals of $S$ generated in degree $d$.
\end{thm}
\begin{proof}
  First of all, we observe that for $f=(y_1-y_2)(y_3-y_4)^{d-1}$ we have $\ell\circ f=0$ and since $\vert \supp(f)\vert =2d$ we get $\nu(n,d)\leq 2d$. To show the equality, let $I\subset S$ be an artinian monomial ideal. We check that for any $f\in (I^{-1})_d$ where, $\ell\circ f=0$, $\vert \supp(f)\vert \geq 2d$.\\
  Using Theorem \ref{max} above, we get that for some $1\leq i\leq n$
  where $1\leq a_i\leq d-1$ we have
  $\vert \supp(f)\vert \geq (a_i+1)(d-a_i+1)$. Observe that since we
  have $a_i\leq d-1$ we get that
  $(a_i+1)(d-a_i+1)=d(a_i+1)-(a_i-1)(a_i+1)\geq 2d$, which completes
  the proof.

\end{proof}

\section{Simplicial complexes and Matroids}
In \cite{GIV} Gennaro, Ilardi and Vall\`es describe a relation between
the failure of the SLP of artinian ideals and the existence of special
singular hypersurfaces. In particular, for the ideals we consider in
this section they proved that in the following cases the ideal $I$
fails the SLP at the range $k$ in degree $d+i-k$ if and only if there
exists at any point $M$ a hypersurface of degree $d+i$ with
multiplicity $d+i-k+1$ at $M$ given by a form in $(I^{-1})_{d+i}$, see
\cite{GIV} for more details. In \cite[Theorem 6.2]{GIV}, they provide
a list of monomial ideals $I\subset S=\mathbb{K}[x_1,x_2,x_3]$
generated in degree $5$ failing the WLP. Here we give the exhaustive
list of such ideals.
\begin{dfn}
  Let $I\subset S$ be an artinian monomial ideal and
  $G =\lbrace m_1,\dots ,m_r\rbrace\subset R_d$ be a monomial
  generating set of $(I^{-1})_d$. Assume that $I$ fails the WLP by
  failing surjectivity in degree $d-1$ thus there is a non-zero
  polynomial $f\in (I^{-1})_d$ with $\supp(f)\subset G$ such that
  $(x_1+\cdots +x_n)\circ f=0$. We say $I$ fails the WLP
  \textit{minimally} if the set $G$ is minimal with respect to
  inclusion.
\end{dfn}
\begin{rem}
  Note that for every artinian monomial ideal $I\subset S$ where the
  WLP fails minimally, there is a unique form in the kernel of the map
  $\circ (x_1+\cdots +x_n )\colon (I^{-1})_d\longrightarrow
  (I^{-1})_{d-1}$.
  In fact, if there are two different forms with the same support we
  can eliminate at least one monomial in one of the forms and get a
  form where its support is strictly contained in the support of the
  previous ones, contradicting the minimality.
\end{rem}
\begin{prop}\label{class 1}
  For an artinian monomial ideal $I\subset S$ generated in degree $5$
  with at least $6$ generators, $S/I$ fails the WLP by failing
  surjectivity in degree $4$ if and only if the set of generators for
  the inverse system module $I^{-1}$ contains the monomials in the
  support of one of the following forms, up to permutation of
  variables:
  \begin{itemize}
  \item $(y_2-y_3)(y_1-y_3)^2(y_1-y_2)(2y_1-y_2-y_3)$
  \item $(y_2-y_3)(y_1-y_3)(y_1-y_2)^2(2y_1+y_2-3y_3)$
  \item
    $(y_2-y_3)(y_1-y_3)(y_1-y_2)(y^2_1+y_1y_2+y^2_2-3y_1y_3-3y_2y_3+3y^2_3)$
  \item
    $(y_2-y_3)(y_1-y_3)(y_1-y_2)(y^2_1-y_1y_2-y^2_2-y_1y_3+3y_2y_3-y^2_3)$
  \item $(y_2-y_3)^2(y_1-y_3)^2(y_1-y_2)$
  \item $(y_2-y_3)(y_1-y_3)(y_1-y_2)^3$
  \item
    $(y_2-y_3)(y_1-y_3)(y_1-y_2)(y^2_1-y_1y_2+y^2_2-y_1y_3-y_2y_3+y^2_3).$
  \end{itemize}
  Moreover, the support of all the above forms define monomial ideals
  failing surjectivity minimally.
\end{prop}
\begin{proof}
  We prove the statement using Macaulay2 and considering all artinian
  monomial ideals generated in degree $5$ with at least $6$
  generators. There are $816$ of such ideals but considering the ones
  failing the WLP by failing surjectivity in degree $4$ and
  considering the forms in the inverse system module $(I^{-1})_5$
  there are only $25$ distinct non-zero forms $f\in (I^{-1})_5$ such
  that $(x_1+x_2+x_3)\circ f=0$. Therefore, every ideal where $I^{-1}$
  contains the support of each polynomial fails WLP by failing
  surjectivity in degree $4$. Permuting the variables we get only $7$
  equivalence classes which correspond to the forms given in the
  statement.
\end{proof}
\begin{rem}
  The support of the last three forms in Proposition \ref{class 1}
  consists of $12$ monomials which is the same as $\nu(3,5)=12$ given
  in \ref{3d-3}. Therefore, the support of each form in the last three
  cases, up to permutations of the variables generates $I^{-1}$ with
  lease possible number of generators in degree $5$ where $I$ fails
  the WLP. \\Using Proposition \ref{factor}, each of the forms above
  factors in linear form over an algebraically closed field;
  e.g. $\mathbb{K}=\mathbb{C}$.
\end{rem}
The next result completely classifies monomial ideals
$I\subset S=\mathbb{K}[x_1,x_2,x_3,x_4]$, generated in degree $3$,
failing the WLP which extends Proposition $6.3$ in \cite{GIV}.
\begin{prop}\label{class}
  For an artinian monomial ideal $I\subset S$ generated in degree $3$
  with at least $10$ generators, surjectivity of the multiplication
  map by a linear form in degree $2$ of $S/I$ fails if and only if the
  set of generators for inverse system module $I^{-1}$ contains the
  monomials in the support of one of the following forms, up to
  permutation of variables:
  \begin{itemize}
  \item $(y_2-y_4)^2(y_1-y_3)$
  \item $(y_2-y_4)(y_1-y_4)(y_1-y_2)$
  \item $(y_2-y_3)(y_1-y_4)(y_1-2y_3+y_4)$
  \item $(y_2-y_3)(y_1-y_4)(y_1-y_2-y_3+y_4)$
  \item $(y_3-y_4)(y_2-y_4)(y_1-y_3)$
  \item $(y_1-y_4)(y_1y_2+y_1y_3-2y_2y_3-2y_1y_4+y_2y_4+y_3y_4)$
  \item $(y_1-y_2)(y_1y_2-y_1y_3-y_2y_3+2y_3y_4-y^2_4)$
  \item $(y_3-y_4)(y^2_2-y_1y_3+y_1y_4-2y_2y_4+y_3y_4)$
  \item $(y_3-y_4)(y^2_1+y^2_2-2y_1y_3-2y_2y_4+2y_3y_4)$
  \item
    $2y^2_1y_2-3y_1y^2_2+2y^2_2y_3-y_1y^2_3-2y^2_1y_4+2y_1y_2y_4+y^2_2y_4+2y_1y_3y_4-4y_2y_3y_4+y^2_3y_4$
  \item
    $y^2_1y_2-y_1y^2_2+y^2_2y_3-y_1y^2_3-y^2_1y_4+2y_1y_3y_4-2y_2y_3y_4+y^2_3y_4+y_2y^2_4-y_3y^2_4$
  \item
    $y^2_1y_2-y^2_1y_3-2y_1y_2y_3+2y_2y^2_3+4y_1y_3y_4-2y_2y_3y_4-2y^2_3y_4-2y_1y^2_4+y_2y^2_4+y_3y^2_4$
  \item
    $y^2_1y_2-y^2_1y_3-y_1y_2y_3+y_2y^2_3-y_1y_2y_4+3y_1y_3y_4-y_2y_3y_4-y^2_3y_4-y_1y^2_4+y_2y^2_4.$
  \end{itemize}
  Moreover, the support of all the above forms define monomial ideals
  failing surjectivity minimally.
\end{prop}
\begin{proof}
  We prove it using the same method as the proof of Proposition
  \ref{class 1} using Macaulay2. There are $8008$ artinian monomial
  ideals generated in degree $3$ with at least $10$ generators.
  Considering the forms in the inverse system module $(I^{-1})_3$
  where $(x_1+x_2+x_3+x_4)\circ f=0$ correspond to the ideals failing
  WLP with failing surjectivity in degree $2$, there are $237$
  distinct non-zero forms. Thus any ideal $I$ where its inverse system
  module $I^{-1}$ contains the support of each of the forms fails WLP
  in degree $2$. Also considering the permutation of the variables
  there are $13$ distinct forms given in the statement.
\end{proof}
\begin{rem}
  The first two forms have $6$ monomials which is the same as
  $\nu(4,3)=6$ given in \ref{2d}. Therefore, each form in the last two
  cases, up to permutation of variables give the minimal number of
  generators for the inverse system module $I^{-1}$ where $I$ fails
  the WLP.  \\One can check that the factors in the forms given in
  Proposition \ref{class} are irreducible even over the complex
  numbers (or any algebraically closed field of characteristic zero).
\end{rem}

The above results lead us to correspond simplicial complexes to the
class of ideals failing the WLP by failing surjectivity.  Recall that
Theorem \ref{3d-3} and Corollary \ref{2d} imply that in the polynomial
ring $S=\mathbb{K}[x_1,\dots,x_n]$ when the Hilbert function of an
artinian monomial algebra generated in a single degree $d$,
$H_{S/I}(d)$, is less than $\nu(n,d)$, the monomial algebra $S/I$
satisfies the WLP. First we recall the following definitions:

\begin{dfn}
  A \textit{matroid} is a finite set of elements $M$ together with the
  family of subsets of $M$, called independent sets, satisfying,
  \begin{itemize}
  \item The empty set is independent,
  \item Every subset of an independent set is independent,
  \item For every subset $A$ of $M$, all maximal independent sets
    contained in $A$ have the same number of elements.
  \end{itemize}
\end{dfn}
A \textit{simplicial complex} $\Delta$ is a set of simplices such that
any face of a simplex from $\Delta$ is also in $\Delta$ and the
intersection of any two simplices is a face of both. Note that every
matroid is also a simplicial complex with independent sets as its
simplices.
\begin{dfn}
  Recall $\mathcal{M}_d$ from Definition \ref{def:deg/level} which is
  the set of monomials in degree $d$ in the ring $R$ and define
  ${{\mathcal{M}}^\prime}_d= {\mathcal{M}}_d\setminus \lbrace
  y^d_1,\dots ,y^d_n\rbrace$.
  We define \textit{independent} set $s\subset \mathcal{M}^\prime _d$
  to be the set of monomials such that the set
  $\lbrace (x_1+\dots +x_n)\circ m\mid m\in s\rbrace$ is a linearly
  independent set. A subset $s\subset {{\mathcal{M}}^\prime}_d$ is
  called \textit{dependent} if it is not an independent set. Then
  define $\Delta_{d,sur}$ to be the simplicial complex with the
  monomials in $\mathcal{M}^\prime_d$ as the ground set and all
  independent sets as its faces. Note that $\Delta_{d,sur}$ forms a
  matroid.
\end{dfn}
Any proper subset of the support of each of the forms in \ref{class 1}
and \ref{class} forms an independent set.  Observe that for every
independent set $s$, monomial ideal $I\subset S$ generated by the
$d$-th power of the variables in $S$ and corresponding monomials of
$\mathcal{M}^\prime_d\setminus s$ in $S$ form an artinian ideal $I$,
where $S/I$ satisfies the WLP. Since the ground set of
$\Delta_{d,sur}$ is the subset of monomials $R_d$ the size of an
independent set in bounded from above with the number of monomials in
$R_{d-1}$. Therefore we have
$\dim(\Delta_{d,sur})\leq h_{d-1}(R) - 1$.
\begin{exmp}
  The support of each polynomial given in Proposition \ref{class 1} is
  a minimal non-face of the simplicial complex $\Delta_{5,sur}$ with
  the ground set
  $\mathcal{M}^\prime_5 =\mathcal{M}_5\setminus \lbrace
  y^5_1,y^5_2,y^5_3\rbrace$.
  This simplicial complex has 25 minimal non-faces (considering the
  permutations of variables). $\Delta_{5,sur}$ has 7 minimal non-faces
  of dimension 11, 6 minimal non-faces of dimension 13 and 12 minimal
  non-faces of dimension 14.
\end{exmp}
Similarly we can construct another simplicial complex by the
complement of dependent sets.
\begin{dfn}
  Define $\Delta^*_{d,sur}$ to be the simplicial complex with the
  monomials in $S_d\setminus \lbrace x^d_1,\dots ,x^d_n\rbrace$ as its
  ground set and faces of $\Delta^*_{d,sur}$ are the corresponding
  monomials of $\mathcal{M}^\prime_d\setminus s$ in $S$ where $s$ is a
  dependent set.
\end{dfn}
Observe that artinian algebra $S/I$ where $I$ is generated by the
$d$-th power of the variables in $S$ together with the monomials in a
face of $\Delta^*_{d,sur}$, fails the WLP. In fact the multiplication
map on $S/I$ form degree $d-1$ to $d$ is not surjective.  Theorem
\ref{3d-3} and \ref{2d} imply that every subset
$s\subset \mathcal{M}^\prime_d$ with $\vert s\vert \leq \nu(n,d)$ is
independent. Therefore we have
$\dim(\Delta^*_{d,sur})=\vert S_d \vert -n-\nu(n,d)-1$, the equality
is because the bound is sharp.
\begin{rem}
  Recall that the \textit{Alexander dual} of a simplicial complex $\Delta$ on
  the ground set $V$ is a simplicial complex with the same ground set
  and faces are all the subsets of $V$ where their complements are
  non-faces of $\Delta$.  Observe that $\Delta^*_{d,sur}$ is a
  simplicial complex in $S_d$ and $\Delta_{d,sur}$ is a simplicial
  complex in the Macaulay dual ring $R_d$. Note that for any
  independent set $s\subset \mathcal{M}^\prime_d$ the corresponding
  monomials of the complement $\mathcal{M}^\prime_d\setminus s$ in the
  ring $S$ is not a face of $\Delta^*_{d,sur}$ which implies that
  $\Delta_{d,sur}$ is Alexander dual to $\Delta^*_{d,sur}$.
\end{rem}
We may construct simplicial complexes corresponding to artinian
algebras failing or satisfying injectivity in a certain degree.
\begin{dfn}
  Define $\Delta_{d,inj}$ to be a simplicial complex with the
  monomials in $\mathcal{M}^\prime_d$ as its ground set and faces
  correspond to generators of $(I^{-1})_d$ where $I$ fails injectivity
  in degree $d-1$. \end{dfn}
\begin{rem}
  Recall that all minimal monomial Togliatti systems correspond to
  facets of $\Delta_{d,inj}$. In fact for minimal monomial Togliatti
  system $I$ the inverse system module has the maximum number of
  generators where $I$ fails injectivity in degree $d-1$.
\end{rem}

\section{WLP of ideals fixed by actions of a cyclic Group}
Mezzetti and Mir\'o-Roig in \cite{MM2} studied artinian ideals of the
polynomial ring $\mathbb{K}[x_1,x_2,x_3]$, where $\mathbb{K}$ is an
algebraically closed field of characteristic zero generated by
homogeneous polynomials of degree $d$ invariant under an action of
cyclic group $\mathbb{Z}/d\mathbb{Z}$, for $d\geq 3$ and they proved
that if $\gcd(a_1,a_2,a_3,d)=1$ they define monomial Togliatti systems. In \cite{MM3}, Colarte, Mezzetti, Mir\'o-Roig and Salat consider such ideals in a polynomial ring with at least three
variables. Throughout this section $\mathbb{K}=\mathbb{C}$ and
$S=\mathbb{K}[x_1,\dots ,x_n]$, where $n\geq 3$. Let $d\geq 2$ and
$\xi = e^{{2\pi i}/d}$ to be the primitive $d$-th root of
unity. Consider diagonal matrix
$$M_{a_1,\dots ,a_n} = \begin{pmatrix}
  \xi^{a_1} & 0 & \cdots  & 0\\
  0 & \xi^{a_2} & \cdots & 0\\
  \vdots & \vdots & &\vdots \\
  0 & 0 &\cdots & \xi^{a_n}
\end{pmatrix}
$$
representing the cyclic group $\mathbb{Z}/d\mathbb{Z}$, where
$a_1,a_2,\dots ,a_n$ are integers and the action is defined by
$[x_1,\dots ,x_n]\mapsto[\xi^{a_1}x_1,\dots ,\xi^{a_n}x_n]$. Since
$\xi^d=1$, we may assume that $0\leq a_i\leq d-1$, for every
$1\leq i\leq n$. Let $I\subset S$ be the ideal generated by all the
forms of degree $d$ fixed by the action of $M_{a_1,\dots ,a_n}$.  In
\cite[Theorem $3.1$]{MM2}, Mezzetti and Mir\'o-Roig showed that these
ideals are monomial ideals when $n=3$. Here we state it in general for
all $n\geq 3$ with a slightly different proof.
\begin{lem}\label{mon ideal}
  For integer $d\geq 2$, the ideal
  $I\subset S=\mathbb{K}[x_1,\dots ,x_n]$ generated by all the forms
  of degree $d$ fixed by the action of $M_{a_1,\dots ,a_n}$ is
  artinian and generated by monomials.
\end{lem}
\begin{proof}
  Since $M_{a_1,\dots ,a_n}$ is a monomial action in the sense that
  for every monomial $m$ of degree $d$ we have
  $M^r_{a_1,\dots ,a_n}m=cm$ for each $0\leq r\leq d-1$ and for some
  $c\in \mathbb{K}$. Then if we have a form of degree $d$ fixed by
  $M_{a_1,\dots ,a_n}$, all its monomials are fixed by
  $M_{a_1,\dots ,a_n}$. This implies that $I$ is a monomial
  ideal. Note also that since $\xi^d=1$, all the monomials
  $x^d_1,x^d_2,\dots ,x^d_n$ are fixed by the action of
  $M_{a_1,\dots ,a_n}$ which means $I$ is artinian ideal.
\end{proof}

Using the above result, from now on we take the monomial set of
generators for $I$. Observe that for two distinct primitive $d$-th
roots of unity we get different actions, but the set of monomials
fixed by both actions are the same. Also the action
$M_{a_n+r,\dots ,a_n+r}$ which is obtained by multiplying the matrix
$M_{a_1,\dots ,a_n}$ with a $d$-th root of unity defines the same
action on degree $d$ monomials in $S$.  In \cite{MM2}, Colarte, Mezzetti, Mir\'o-Roig and Salat show
that in the case that $n=3$ where $a_i$'s are distinct and
$\gcd(a_1,a_2,a_3,d)=1$, these ideals are all monomial Togliatti
systems. In fact they show that the WLP of these ideals fails in
degree $d-1$ by failing injectivity of the multiplication map by a
linear form in that degree. In this section, we study the cases where
WLP of such ideals fail by failing surjectivity in degree $d-1$. Then
we classify all such ideals in polynomial rings with more than $2$
variables, in terms of their WLP.

We start this section by stating some results about the number of
monomials of degree $d$ fixed by the action $M_{a_1,\dots ,a_n}$ of
$\mathbb{Z}/d\mathbb{Z}$ in $S$. In fact we prove that this number
depends on the integers $a_i$'s. In the next result we give an
explicit formula computing the number of such monomials where $n=3$.
\begin{prop}\label{galois mons}
  For integers $a_1,a_2,a_3$ and $d\geq 2$, the number of monomials in
  $ S=\mathbb{K}[x_1,x_2 ,x_3]$ of degree $d$ fixed by the action of
  $M_{a_1,a_2,a_3}$ is

\begin{equation}\label{no.}
  1+\frac{\gcd(a_2-a_1,a_3-a_1,d)\cdot d+\gcd(a_2-a_1,d)+\gcd(a_3-a_1,d)+\gcd(a_3-a_2,d)}{2}.
\end{equation}
\end{prop}
\begin{proof}
  From the discussion above, the number of monomials of degree $d$
  fixed by $M_{0,a_2-a_1,a_3-a_1}$ and $M_{a_1,a_2,a_3}$ are the
  same. Thus, we count the number of monomials of degree $d$ fixed by
  $M_{0,a_2-a_1,a_3-a_1}$. Any monomial of degree $d$ in $S$ can be
  written as $x^{d-m-n}y^mz^n$ with $0\leq m,n\leq d$ and $m+n\leq d$
  and it is invariant under the action of $M_{0,a_2-a_1,a_3-a_1}$ if
  and only if $(a_2-a_1)m+(a_3-a_1)n\equiv 0\pmod d$. In \cite[Chapter
  3]{McCarthy}, we find that the number of congruent solutions of
  $(a_2-a_1)m+(a_3-a_1)n\equiv 0\pmod d$ is
  $\gcd(a_2-a_1,a_3-a_1,d)\cdot d$ but since the solutions $(0,0)$,
  $(0,d)$ and $(d,0)$ (corresponding to the powers of variables) are
  all congruent to $d$ and fixed by $M_{0,a_2-a_1,a_3-a_1}$ we get two
  more solutions than $\gcd(a_2-a_1,a_3-a_1,d)\cdot d$. In order to
  count the monomials of degree $d$ invariant under the action of
  $M_{0,a_2-a_1,a_3-a_1}$ we need to count the number of solutions of
  $(a_2-a_1)m+(a_3-a_1)n\equiv 0\pmod d$ satisfying the extra
  condition $m+n\leq d$.

  First we count the number of such solutions when $m=0$ and
  $n\neq 0$. So every $1\leq n< \gcd(a_3-a_1,d)$ is a solution of
  $(a_3-a_1)n\equiv 0\pmod d$. Therefor there are $\gcd(a_3-a_1,d)-1$
  solutions in this case.  Similarly, there are $\gcd(a_2-a_1,d)-1$
  solutions when $n=0$ and $m\neq 0$.  Counting the solutions when
  $m+n=d$ is equivalent to counting the solutions of
  $(a_3-a_2)m\equiv 0\pmod d$ which is similar to the previous case
  and is equal to $\gcd(a_3-a_2,d)-1$. There is also one solution when
  $m=n=0$.

  Now rest of the solutions (where $m\neq 0$ and $n\neq 0$ and
  $m+n\neq d$) by \cite[Chapter $3$]{McCarthy} is equal to
  $\gcd (a_2-a_1,a_3-a_1,d)\cdot d - \gcd(a_3-a_2,d)
  -\gcd(a_2-a_1,d)-\gcd(a_3-a_1,d) +2$
  but we need to count the number of those satisfying $0<m+n<d$. Note
  that if $0< m_0 < d$ and $0< n_0 <d$ is a solution of
  $(a_2-a_1)m+(a_3-a_1)n\equiv 0\pmod d$ then $0<d-m_0< d$ and
  $0< d-n_0 <d$ is also a solution but one and only one of the two
  conditions $0< m_0+n_0< d$ and $0< d-m_0+d-n_0< d$
  is satisfied. Therefore, there are
  $$\frac{\gcd (a_2-a_1,a_3-a_1,d)\cdot d - \gcd(a_3-a_2,d)
    -\gcd(a_2-a_1,d)-\gcd(a_3-a_1,d) +2}{2}$$
  solutions satisfying $0<m+n<d$.  Adding this with the solutions
  where $m=0$ or $n=0$ or $m+n=d$ which we have counted them above
  together with two more pairs $(0,d)$ and $(d,0)$ (explained in the
  beginning of the proof) we get what we wanted to prove.
\end{proof}
For a fixed integer $d\geq 2$ Proposition \ref{galois mons} shows that
how the number of fixed monomials of degree $d$ depends on the
integers $a_1,a_2,a_3$. In the following example we see how they are
distributed.
\begin{exmp}
  Using Formula (\ref{no.}) we count the number of monomials of degree
  $15$ in $\mathbb{K}[x_1,x_2,x_3]$ fixed by the action $M_{0,a,b}$
  for every $0\leq a,b\leq 14$. We see the distribution of them in
  terms of $\mu(I)$ in the following table:
  \begin{center}
    \begin{tabular}{|c|c|c|c|c|c|c|c|c|c|c|}
      \hline $m$ & 10 & 11 & 12& 13& 17&28&34&46&51&136\\
      \hline $d_{m}$ & 24 & 72 & 24& 48 & 24 & 12 & 12 & 2& 6 & 1\\
      \hline 
    \end{tabular}
  \end{center}
  where 	$d_{m}=\vert \lbrace (a,b)\mid \mu(I)=m\rbrace\vert $.  Note
  that the last column of the table corresponds to the action $M_{0,0,0}$
  where we get $\mu(I)=(\mathbb{K}[x_1,x_2,x_3])_{15}=136$. There are
  exactly $24$ pairs $(a,b)$ where either at least one of them is zero
  or $a=b$, which in these cases we get $\mu(I)=17$. We have
  $\gcd(a,b,d)\neq 1$ for all the cases with $\mu(I)>17$ and
  $\gcd(a,b,d)= 1$ for all the cases with $\mu(I)<17$.
\end{exmp}
As we saw in the above example the distribution of the number of
monomials of degree $d$ fixed by $M_{a_1,a_2,a_3}$ is quite difficult
to understand but we prove that such numbers are bounded from above
depending on the prime factors of $d$ in the case that $a_i$'s are
distinct and $\gcd(a_1,a_2,a_3,d)=1$ .
\begin{prop}
  For $d\geq 3$ and distinct integers $0\leq a_1,a_2,a_3\leq d-1$ with
  $\gcd(a_1,a_2,a_3,d)=1$, let $\mu(I)$ be the number of monomials of
  degree $d$ fixed by $M_{a_1,a_2,a_3}$. Then
  \[ \mu(I) \leq \left\{ \begin{array}{ll}
                           \frac{(p+1)d+p^2+3p}{2p} & \mbox{if $p^2\nmid d$}\\
                           \frac{(p+1)d+4p}{2p} & \mbox{if
                                                  $p^2\mid d,$}
                         \end{array} \right. \] 
                       where $p$ is the smallest prime dividing $d$. Moreover, the bounds are sharp.
                     \end{prop}
                     \begin{proof}
                       Using Proposition \ref{galois mons} we provide
                       an upper bound for
                       $\gcd(a_2-a_1,d)+\gcd(a_3-a_1,d)+\gcd(a_3-a_2,d)$.
                       For some integer $t$ we have
                       $d=\gcd(a_2-a_1,d)\cdot\gcd(a_3-a_1,d)\cdot\gcd(a_3-a_2,d)\cdot
                       t$.
                       Since
                       $\gcd(a_3-a_1,d)\cdot
                       \gcd(a_3-a_2,d)=\frac{d}{\gcd(a_2-a_1,d)\cdot
                         t}$, we have
                       $$\gcd(a_3-a_1,d) + \gcd(a_3-a_2,d)\leq
                       1+\frac{d}{\gcd(a_3-a_2,d)\cdot t}.$$ Therefore,
                       \begin{align*}
                         \gcd(a_2-a_1,d)+\gcd(a_3-a_1,d)+\gcd(a_3-a_2,d)\leq & \gcd(a_2-a_1,d)+\frac{d}{\gcd(a_3-a_1,d)\cdot t}+1\\ \leq  &\gcd(a_2-a_1,d)+\frac{d}{\gcd(a_2-a_1,d)}+1\\ \leq & p+\frac{d}{p}+1.
                       \end{align*}
                       Note that, $\gcd(a_2-a_1,d)+\gcd(a_3-a_1,d)+\gcd(a_3-a_2,d)=d+2> p+\frac{d}{p}+1$ if and only if at least two integers $a_i$ are the same which contradicts the assumption. Since for every $q\geq p$ we have $p+\frac{d}{p}+1\geq q+\frac{d}{q}+1$ we get $\gcd(a_2-a_1,d)+\gcd(a_3-a_1,d)+\gcd(a_3-a_2,d)\leq p+\frac{d}{p}+1$.\\
                       Now assume that $p^2\nmid d$, to reach the
                       bound we let $a_2-a_1=p$ and
                       $a_3-a_1=\frac{d}{p}$. In this case since
                       we have that $\gcd(a_3-a_2,d)=1$, Proposition \ref{galois
                         mons} implies that
                       $\mu(I)\leq \frac{(p+1)d+p^2+3p}{2p}$.  \\If
                       $p^2\mid d$, choosing $a_2-a_1=p$ and
                       $a_3-a_1=\dfrac{d}{p}$ implies that
                       $\gcd(a_3-a_2,d)=p$. So the given bound can not
                       be sharp. Observe that for $q> p$ and $q\mid d$
                       we have $q+\frac{d}{q}+1\leq 1+\frac{d}{p}+1$.
                       Therefore in this case we have
                       $\gcd(a_2-a_1,d)+\gcd(a_3-a_1,d)+\gcd(a_3-a_2,d)\leq
                       1+\frac{d}{p}+1$,
                       and equality holds for $a_2-a_1=1$ and
                       $a_3-a_1=\frac{d}{p}$ so by Proposition
                       \ref{galois mons} we have that
                       $\mu(I)\leq \frac{(p+1)d+4p}{2p}$.
                     \end{proof}
                     In the proof of Proposition \ref{galois mons}, we
                     used the fact that the number of solutions
                     $(m,n)$ for
                     $(a_2-a_1)m+(a_3-a_1)n\equiv 0\pmod d$
                     (corresponding to the action by
                     $M_{a_1,a_2,a_3}$) where $m,n\neq 0$ and
                     $m+n\neq d$ is exactly twice the number of
                     solutions of $(b-a)m+(c-a)n\equiv 0\pmod d$
                     satisfying $0<m+n<d$. But in the polynomial ring
                     with more than three variables this is no longer
                     the case that the solutions of the corresponding
                     equation of $M_{a_1,\dots ,a_n}$ are distributed
                     in a nice way so we do not have the explicit
                     formula as in Proposition \ref{galois mons} in
                     higher number of variables.  In
                     Proposition \ref{g4} below we provide an upper
                     bound for this number in the polynomial ring with
                     four variables where
                     $\gcd(a_1,a_2,a_3,a_4,d)=1$. The bound implies
                     $H_{S/I}(d-1)\leq H_{S/I}(d)$ and therefore the
                     WLP in degree $d-1$ is an assertion of
                     injectivity. In \cite[Theorem $4.8$]{MM3}, Colarte, Mezzetti, Mir\'o-Roig and Salat show that the number of monomials in $(\mathbb{K}[x_1,\dots ,x_n])_{n+1}$ fixed  by the action $M_{0,1,2,\dots ,n}$ of $\mathbb{Z}/({n+1})\mathbb{Z}$ is bounded above by $\binom{2n-1}{n-1}$, for any $n\geq 3$.
                     \begin{prop}\label{g4}
                       For $d\geq 2$ and integers
                       $0\leq a_1,a_2,a_3,a_4\leq d-1$, where at most
                       two of the integers among $a_i$'s are equal and
                       $\gcd(a_1,a_2,a_3,a_4,d)=1$. Let $\mu(I)$ be
                       the number of monomials of degree $d$ in
                       $S=\mathbb{K}[x_1,x_2,x_3,x_4]$ fixed by
                       $M_{a_1,a_2,a_3,a_4}$. Then
$$
\mu(I)\leq 1+\frac{(d+2)(d+1)}{2}.
$$ 
\end{prop}
\begin{proof}
  Any monomial of degree $d$ in $S$ can be written as
  $x^{m_1}_1x^{m_2}_2x^{m_3}_3x^{d-m_1-m_2-m_3}_4$ with
  $0\leq m_1 , m_2,m_3\leq d$ and $m_1+m_2+m_3\leq d$. Monomial
  $x^{m_1}_1x^{m_2}_2x^{m_3}_3x^{d-m_1-m_2-m_3}_4$ is invariant under
  the action of $M_{a_1,a_2,a_3,a_4}$ or equivalently
  $M_{a_1-a_4,a_2-a_4,a_3-a_4,0}$ if and only if
  \begin{equation}\label{equa}
    (a_1-a_4)m_1+(a_2-a_4)m_2+(a_3-a_4)m_3\equiv 0\pmod d,  \hspace*{0.5 cm}  m_1+m_2+m_3\leq d.
  \end{equation}In \cite[Chapter $3$]{McCarthy} we find that the number of congruent solutions of $(a_1-a_4)m_1+(a_2-a_4)m_2+(a_3-a_4)m_3\equiv 0\pmod d$ is $d^2$. We first count the number of congruent solutions of \ref{equa} where at least one of $m_1,m_2$ or $m_3$ is zero. Suppose $m_1=0$ then by \cite[Chapter $3$]{McCarthy}, the number of congruent solutions of $(a_2-a_4)m_2+(a_3-a_4)m_3\equiv 0\pmod d$ is $gcd(a_2-a_4,a_3-a_4,d)\cdot d$. Similarly, by \cite[Chapter $3$]{McCarthy}, the number of congruent solutions of \ref{equa} having two coordinates  zero, for example $m_1=m_2=0$, is $gcd(a_3-a_4,d)$. All together the number of congruent solutions of \ref{equa} where at least one of the coordinates $m_1,m_2,m_3$ is zero is as follows 
  \begin{align*}
    & d\left( gcd(a_1-a_4,a_2-a_4,d)+gcd(a_2-a_4,a_3-a_4,d)+gcd(a_1-a_4,a_3-a_4,d)\right)\\
    & -gcd(a_1-a_4,d)-gcd(a_2-a_4,d)-gcd(a_3-a_4,d)+1.
  \end{align*}
  Note that if $(m_{10},m_{20},m_{30})$ is a solution of \ref{equa}
  such that $m_{i0}\neq 0$ for $i=1,2,3$, then
  $(d-m_{10},d-m_{20},d-m_{30})$ is a solution of
  $(a_1-a_4)m_1+(a_2-a_4)m_2+(a_3-a_4)m_3\equiv 0\pmod d$ where
  $3d-m_{10}-m_{20}-m_{30}\geq d$. Therefor, the number of congruent
  solutions of \ref{equa} where no $m_i$ is zero is bounded from above
  by
  \begin{align*}
    [d^2-(d&(gcd(a_1-a_4,a_2-a_4,d)+gcd(a_2-a_4,a_3-a_4,d)+gcd(a_1-a_4,a_3-a_4,d))\\
           &-gcd(a_1-a_4,d)-gcd(a_2-a_4,d)-gcd(a_3-a_4,d)+1)]/2.
  \end{align*}
  Using Proposition \ref{galois mons}, we count the number of
  solutions \ref{equa} where at least one of the coordinates $m_i$ is
  zero. If $m_1=0$ then by Proposition \ref{galois mons} the number of
  solutions of $(a_2-a_4)m_2+(a_3-a_4)m_3=0\pmod d$ where
  $0\leq m_2,m_3\leq d$ and $m_2+m_3\leq d$ is
$$
\frac{\gcd(a_2-a_4,a_3-a_4,d)\cdot
  d+\gcd(a_2-a_4,d)+\gcd(a_3-a_4,d)+\gcd(a_2-a_3,d)+2}{2}.
$$
Similarly we can count the number of such solutions when $m_2=0$ or
$m_3=0$. Now suppose that $m_1=m_2=0$ then we get $\gcd(a_3-a_4,d)+1$
where $0\leq m_3\leq d$. All together the number of solutions of
\ref{equa} where at least one $m_i$ is zero is
\begin{align*}
  [&d(\gcd(a_1-a_4,a_2-a_4,d)+\gcd(a_2-a_4,a_3-a_4,d)+\gcd(a_1-a_4,a_3-a_4,d))\\
   &+\gcd(a_1-a_2,d)+\gcd(a_1-a_3,d)+\gcd(a_2-a_3,d)+2]/2.
\end{align*}
Therefore, the number of solutions of \ref{equa} is bounded from above
by
\begin{equation}\label{solutions}
  \frac{d^2+\sum^3_{i=1}\gcd(a_i-a_4,d)+\sum_{1\leq i<j\leq 3}\gcd(a_i-a_j,d)+1}{2}.
\end{equation}
To show the assertion of the theorem we need to show (\ref{solutions})
is bounded from above by $\frac{(d+2)(d+1)+2}{2}$ where at most 2 of
integers among $a_i$'s are equal and $\gcd(a_1,a_2,a_3,a_4,d)=1$. So
we need to show that
\begin{equation}\label{last}
  \sum^3_{i=1}\gcd(a_i-a_4,d)+\sum_{1\leq i<j\leq 3}\gcd(a_i-a_j,d)=\sum^{}_{1\leq i\leq j\leq 4}\gcd(a_i-a_j,d)\leq 3d+3.
\end{equation}
To show this we consider the following cases:
\begin{itemize}
\item[$(1)$]Suppose at least two terms in the left hand side of
  (\ref{last}) are equal to $d$ then at least three integers among
  $a_i$'s are equal which contradicts the assumption.
\item[$(2)$] Suppose that one of the terms in the left hand side is
  equal to $d$. By relabeling the indices we may assume that
  $\gcd(a_1-a_2,d)=d$, this implies that $a_1=a_2$ then we need to
  show that
$$
d+2\gcd(a_1-a_3,d)+2\gcd(a_1-a_4,d)+\gcd(a_3-a_4,d)\leq 3d+3
$$
since we assume that $\gcd(a_1,a_2,a_3,a_4,d)=1$ we have that
$\gcd(a_1-a_4,d)$, $\gcd(a_3-a_4,d) $ and $\gcd(a_1-a_3,d)$ are all
distinct and strictly less than $d$. Thus we have
$$d+2\gcd(a_1-a_3,d)+2\gcd(a_1-a_4,d)+\gcd(a_3-a_4,d)\leq
d+2\frac{d}{2}+2\frac{d}{3}+\frac{d}{4} <3d+3 .$$
\item[$(3)$] Suppose all the terms in the left hand side of
  (\ref{last}) are strictly less than $d$. Then the assumption
  $\gcd(a_1,a_2,a_3,a_4,d)=1$ implies that at most two terms can be
  $d/2$ and assuming the other terms are $d/3$ we get
$$\sum^{}_{1\leq i\leq j\leq 4}\gcd(a_i-a_j,d)\leq 3d+3\leq 2(d/2)+4(d/3)=d+d/2<3d+3.$$
\end{itemize}

\end{proof}
In the rest of this section we study the WLP of ideals in
$S=\mathbb{K}[x_1,\dots ,x_n]$ for $n\geq 3$ generated by all forms of
degree $d\geq 3$ invariant by the action $M_{a_1,\dots ,a_n}$ of
$\mathbb{Z}/d\mathbb{Z}$. First we prove the following  key lemma.
\begin{lem}\label{form}
  For integer $d\geq 2$ and distinct integers
  $0\leq a_1,a_2,a_3\leq d-1$, let $M_{a_1,a_2,a_3}$ be a
  representation of $\mathbb{Z}/d\mathbb{Z}$. Define the linear form
 $$L = \sum^l_{j=0}\xi^j x_1+\sum^{l+k+1}_{j=l+1}\xi^j x_2+\sum^{2d-1}_{j=l+k+2}\xi^j x_3,$$
 where $l$ and $k$ are the residues of $a_2-a_3-1$ and $a_3-a_1-1$
 modulo $d$.  Then the support of the form $F = L^d-\overline{L}^d$
 are exactly the monomials of degree $d$ in $\mathbb{K}[x_1,x_2,x_3]$
 which are not invariant under the action of $M_{a_1,a_2,a_3}$, where
 $\overline{L}$ is the conjugate of $L$ and $\xi$ is a primitive
 $d$-th root of unity.
\end{lem}
\begin{proof}
  First, note that for a rational number $j$ we let
  $\xi^j = e^{j\frac{2\pi i}{2}}$. We observe that for integers
  $0\leq p\leq q$ we have
  $\sum^{q}_{p}\xi^i =
  \xi^{\frac{p+q}{2}}\sum^{q}_{p}\xi^{i-{\frac{p+q}{2}}}$,
  where
  $\sum^{q}_{p}\xi^{i-{\frac{p+q}{2}}} =
  \xi^{\frac{p-q}{2}}+\xi^{{\frac{p-q}{2}}+1}+\cdots
  +\xi^{\frac{q-p}{2}-1}+\xi^{\frac{q-p}{2}}$
  which is invariant under conjugation, so it is a real
  number. Therefore, we have
 $$L = \sum^l_{j=0}\xi^j x_1+\sum^{l+k+1}_{j=l+1}\xi^j x_2+\sum^{2d-1}_{j=l+k+2}\xi^j x_3= r_1\xi^{\frac{l}{2}} x_1 + r_2\xi^{\frac{2l+k+2}{2}} x_2 + r_3\xi^{\frac{l+k+1}{2}}x_3$$
 where $r_1,r_2$ and $r_3$ are non-zero real numbers. In fact, using
 the assumption that $a_1$, $a_2$ and $a_3$ are distinct we get that
 $0\leq l,k\leq d-2$ which implies that the $r_i$'s are all
 non-zero. The form $F$ can be written as
 \begin{align*}
   F &= L^d-\overline {L}^d\\ &= \left( r_1\xi^{\frac{l}{2}} x_1 + r_2\xi^{\frac{2l+k+2}{2}} x_2 + r_3\xi^{\frac{l+k+1}{2}}x_3\right) ^d - \left( r_1\xi^{-\frac{l}{2}} x_1 + r_2\xi^{-\frac{2l+k+2}{2}} x_2 + r_3\xi^{-\frac{l+k+1}{2}}x_3\right)^d.
 \end{align*}
 Consider monomial $m=x^{\alpha_1}_1 x^{\alpha_2}_2 x^{\alpha_3}_3$ of
 degree $d$ in $\mathbb{K}[x_1,x_2,x_3]$. The coefficient of $m$ in
 $F$ is zero if and only if the coefficients of $m$ in $L^d$ is
 real. The coefficient of $m$ in $L^d$ is real if and only if
 $$\alpha_1\frac{l}{2}+\alpha_2\frac{2l+k+2}{2}+\alpha_3\frac{l+k+1}{2}\equiv
 \alpha_1\frac{- l}{2}+\alpha_2\frac{-(2l+k+2)}{2}+\alpha_3
 \frac{-(l+k+1)}{2} \pmod d$$ which is equivalent to have
 $$\alpha_1 l+\alpha_2 (2l+k+2)+\alpha_3 (l+k+1)\equiv 0, \pmod d.$$
 Therefore, the monomials with non-zero coefficients in $F$ are
 exactly the monomials of degree $d$ in $\mathbb{K}[x_1,x_2,x_3]$,
 which are not fixed by the action of
 $M_{l,{2l+k+2},{l+k+1}}$. Substituting $l,k$ we get that
 $M_{l,{2l+k+2},{l+k+1}}$ is equivalent to the action
 $M_{a_2-a_3-1,2a_2-a_3-a_1-1,a_2-a_1-1}$ and by adding the indices
 with $a_1-a_2+a_3+1$ the last one is also equivalent to
 $M_{a_1,a_2,a_3}$ which proves what we wanted.
\end{proof}
\begin{rem}
  The assumption in Lemma \ref{form} that $a_i$'s are distinct is
  necessary to have the form $F$ non-zero. If at least two of the
  integers $a_i$ are equal then in the linear form $L$ at least the
  coefficient of one of the variables $x_1,x_2$ and $x_3$ is
  zero. Then we conclude that in $L^d$ all the monomials have real
  coefficients which implies $F=0$.

  Lemma \ref{form}, can be extended to any polynomial ring with odd
  number of variables. In fact in this case we can find $n-1$ integers
  $l_i$ in terms of the integers $a_i$ defining the action
  $M_{a_1,\dots ,a_n}$ in such a way that a similar linear form as $L$
  in the lemma in $n$ variables does the same.
\end{rem}
In \cite[Proposition $4.6$]{MM3}, Colarte, Mezzetti, Mir\'o-Roig and Salat show that
the WLP of $I$ fails by failing injectivity in degree $d-1$ in the
polynomial ring $\mathbb{K}[x_1,\dots ,x_n]$. In fact they provide the non-zero
form $f=\prod^{d-1}_{i=1}(\xi^{ia_1}x_1+\cdots +\xi^{ia_n}x_n) $ in the kernel of the multiplication map by a linear form on
artinian algebra $\mathbb{K}[x_1,\dots ,x_n]/I$ from degree $d-1$ to
degree $d$. So all the monomials with non-zero coefficient in
 $(x_1+\cdots+x_n)f$ are fixed by the action $M_{a_1,\dots ,a_n}$.
\label{form ker}
  
We can now state and prove our main theorem which generalizes \cite[Proposition $3.2$]{MM2} and \cite[Proposition $4.6$]{MM3} and gives the complete classification of ideals in $S=\mathbb{K}[x_1,\dots ,x_n]$ generated by all forms of degree $d$ fixed by the action of $M_{a_1,\dots ,a_n}$, for every $n\geq 3$ and $d\geq 2$, in terms of their WLP. \\

\begin{thm}\label{sur}
  For integers $d\geq 2$, $n\geq 3$ and
  $0\leq a_1,\dots ,a_n\leq d-1$, let $M_{a_1,\dots ,a_n}$ be a
  representation of cyclic group $\mathbb{Z}/d\mathbb{Z}$ and
  $I\subset S=\mathbb{K}[x_1,\dots ,x_n]$ be the ideal generated by
  all forms of degree $d$ fixed by the action of $M_{a_1,\dots ,a_n}$.
  Then, $I$ satisfies the WLP if and only if at least $n-1$ of the
  integers $a_i$ are equal.
\end{thm}
\begin{proof}
  Suppose at least $n-1$ of the integers $a_i$'s are equal and by
  relabeling the variables we may assume that
  $a_1=a_2=\dots =a_{n-1}$. For $n=3$, Lemma $5.2$ \cite{MM2} shows that $I$ satisfies the WLP. Similarly for $n\geq 3$ the ideal $I$ contains 
  $(x_1,x_2, \dots ,x_{n-1})^d$, and then all the monomials in $(S/I)_d$
  are divisible by $x_n$ which implies that the map
  $\times x_n :(S/I)_{d-1}\longrightarrow (S/I)_d$ is
  surjective. Since $[(S/I)/x_n(S/I)]_d=0$ we have that
  $[(S/I)/x_n(S/I)]_j=0$ for all $j\geq d$ and then
  $\times x_n :(S/I)_{j-1}\longrightarrow (S/I)_j$ is surjective for
  all $j\geq d$. On the other hand, since $I$ is generated in degree
  $d$, the map $\times x_n :(S/I)_{j-1}\longrightarrow (S/I)_j$ is
  injective, for every $j<d$. Therefore, $I$ has the WLP.

  To show the other implication, we assume that at most $n-2$ integers $a_i$ are equal and we prove that $I$ fails WLP  by showing that map $\times (x_1+\cdots+x_n) : (S/I)_{d-1}\longrightarrow (S/I)_d$ is neither injective nor surjective. \\
  By \cite[Proposition $4.6$]{MM3}, for the non-zero form $f=\prod^{d-1}_{i=1}(\xi^{ia_1}x_1+\cdots +\xi^{ia_n}x_n)$ of degree $d-1$ we have that $(x_1+\cdots +x_n)f$ is a form of degree $d$ in $I$. Therefore the map $\times (x_1+\cdots +x_n) : (S/I)_{d-1}\longrightarrow (S/I)_d$ is not injective.\\
  Now it remains to show the failure of surjectivity. To do so by Macaulay duality equivalently we show that the map $\circ (x_1+\cdots +x_n) : (I^{-1})_{d}\longrightarrow (I^{-1})_{d-1}$ is not injective. Note that the inverse module $(I^{-1})_d$ is generated by all the monomials of degree $d$ in the dual ring $R=\mathbb{K}[y_1,\dots ,y_n]$ which are not fixed by the action $M_{a_1,\dots ,a_n}$.\\
  We consider two cases depending on $a_i$'s. First, assume that there
  are at least three distinct integers among $a_i$'s and by relabeling
  the variables we may assume that $a_1<a_2<a_3$.

  By applying Lemma \ref{form} on the ring $R$, we get the linear form
  $$L = \sum^l_{j=0}\xi^j y_1+\sum^{l+k+1}_{j=l+1}\xi^j
  y_2+\sum^{2d-1}_{j=l+k+2}\xi^j y_3,$$
  where $l$ and $k$ are the residues of $a_2-a_3-1$ and $a_3-a_1-1$
  modulo $d$ and $\xi$ is a primitive $d$-th root of unity.  Since
  $a_1,a_2$ and $a_3$ are distinct $F=L^d-\overline{L}^d$ is non-zero
  form of degree $d$. The monomials with non-zero coefficients in $F$
  are exactly the monomials of degree $d$ in $\mathbb{K}[y_1,y_2,y_3]$
  which are not fixed by the action $M_{a_1,a_2,a_3}$. Therefore, all
  the monomials of degree $d$ in $R$ fixed by the action
  $M_{a_1,\dots , a_n}$ have coefficient zero in $F$ and thus we get
  that $F\in (I^{-1})_d$. Moreover, sum of the coefficients in $L$ is
  exactly $2(1+\xi^1+\xi^2+\cdots +\xi^{d-1})=0$. Therefore,
  $(x_1+\cdots +x_n)\circ F = (x_1+\cdots +x_n)\circ (L^d)
  -(x_1+\cdots +x_n)\circ (\overline{L}^d) =0$
  and this implies that
  $\times (x_1+\cdots+x_n) : (S/I)_{d-1}\longrightarrow (S/I)_d$ is
  not surjective in this case.
 
  Now assume that there are only two distinct integers among
  $a_i$'s. Without loss of generality we may assume that
  $a_1=a_2=\cdots =a_{m}< a_{m+1}=a_{m+2}=\cdots =a_n$. Since we
  assume that at most $n-2$ of the integers $a_i$'s are equal, we have
  $m,n-m\geq 2$ and so $a_1=a_2\neq a_{n-1}=a_n$. Consider the element
  $H=(y_1-y_2)(y_n-y_{n-1})^{d-1}\in R$. Acting $M^r_{a_1,\dots ,a_n}$
  on $H$ we get that
  $M^r_{a_1,\dots
    ,a_n}(y_1-y_2)(y_n-y_{n-1})^{d-1}=\xi^{r{a_1-a_n}}(y_1-y_2)(y_n-y_{n-1})^{d-1}$
  for every $0\leq r\leq d-1$. So $H$ is fixed by the action
  $M_{a_1,\dots ,a_n}$ if and only if $a_1=a_n$ which we assumed
  $a_1\neq a_n$. This implies that $H$ and none of the monomials in
  $H$ are fixed by $M_{a_1,\dots ,a_n}$, therefore $H\in (I^{-1})_d$.
  Moreover, we have that $(x_1+\cdots +x_n)\circ H =0$ and then the
  map $\times (x_1+\cdots+x_n) : (S/I)_{d-1}\longrightarrow (S/I)_d$
  is not surjective.
\end{proof}
We illustrate Theorem \ref{sur} in the next example for the ideal in
the polynomial ring with three variables failing the WLP.
\begin{exmp}
  Let $I\subset  S = \mathbb{K}[x_1,x_2,x_3]$ be the ideal generated by forms of degree $10$ fixed by the action of $M_{0,2,4}$ Theorem \ref{mon ideal} implies that $I$ is generated by all monomials of degree $d$ fixed by the action of $M_{0,2,4}$. By Theorem \ref{sur} above we get that $I$ fails WLP form degree $9$ to degree $10$. Since by Theorem \ref{galois mons} we have $H_{S/I}(10)=52<55=H_{S/I}(9)$, failing WLP is an assertion of failing surjectivity of the multiplication map $\times (x_1+x_2+x_3):(S/I)_{9}\longrightarrow(S/I)_{10}$. We equivalently show that the map $\circ (x_1+x_2+x_3):(I^{-1})_{10}\longrightarrow (I^{-1})_9$ is not injective.\\
  Using Lemma \ref{form}, we let $L$ be the linear form
  $L =
  \sum^7_{j=0}\xi^jy_1+\sum^{11}_{j=8}\xi^jy_2+\sum^{19}_{j=12}\xi^jy_3$
  for $l=7$ and $k=3$ in the dual ring
  $R=\mathbb{K}[y_1,y_2,y_3]$. Then we get the non-zero form
  $F=L^{10}-\overline{L}^{10}$ in the kernel of the map
  $\circ (x_1+x_2+x_3):(I^{-1})_{10}\longrightarrow (I^{-1})_9$.
  Computations by Macaulay2 software, show that the kernel of this
  map has dimension 2. We can actually get the other form in the
  kernel by changing $\xi$ with $\xi^\prime =\xi^3 = e^{{6\pi i}/d}$.
  Therefore we have
  $L^\prime =
  \sum^7_{j=0}\xi^{3j}y_1+\sum^{11}_{j=8}\xi^{3j}y_2+\sum^{19}_{j=12}\xi^{3j}y_3$
  and then $G = {L^\prime}^d-\overline{L^\prime}^d$ is another form of
  degree $10$ in the kernel where $(x_1+x_2+x_3)\circ G =0$.
\end{exmp}

\section{Dihedral Group acting on $\mathbb{K}[x,y,z]$ }
In the previous section we have studied the WLP of ideals generated by
invariant forms of degree $d$ under an action of cyclic group of order
$d$. In this section we study an action of dihedral group $D_{2d}$ on
the polynomial ring with three variables $S=\mathbb{K}[x,y,z]$ where
$\mathbb{K}=\mathbb{C}$ and $d\geq 2$. Let $\xi^{{2\pi i}/d}$ be a
primitive $d$-th root of unity and
\begin{align*} A_d=
  \begin{pmatrix}
    \xi & 0 &0 \\0 & \xi^{-1} & 0 \\0 & 0 &1&
  \end{pmatrix}, B_d=
                                              \begin{pmatrix}
                                                0 & \xi^{-1}   &0  \\
                                                \xi &  0 &  0 \\
                                                0 & 0 &-1&
                                              \end{pmatrix}
\end{align*}
be a representation of dihedral group $D_{2d}$.  Let
$F=\prod^{d-1}_{j=0}(\xi^jx+\xi^{-j}y+z)(\xi^{j}x+\xi^{-j}y-z)$ which
is a polynomial of degree $2d$ invariant by the action $A_d$ and $B_d$
of dihedral group $D_{2d}$. We study the WLP of the artinian monomial
ideal in $S$ generated by all the monomials in $F$ with non-zero
coefficients. First we count the number of generators of such ideals.

\begin{prop}\label{dihedral}
  For integer $d\geq 2$, let $A_d$ and $B_d$ be a representation of
  $D_{2d}$ and let $I\subset S$ be the artinian monomial ideal
  generated by all monomial with non-zero coefficients in
  $F = \prod^{d-1}_{j=0}(\xi^jx+\xi^{-j}y+z)(\xi^{j}x+\xi^{-j}y-z)$.
  Then $\mu(I) =d+3$, if $d=2k+1$; and $\mu(I) =2d+5$, if $d=2k$.
\end{prop}
\begin{proof}
  Fist, assume $d=2k+1$ and consider the action of
  $M_{2,2d-2,d}=\begin{pmatrix}
    \omega^2 &0& 0\\
    0 &\omega^{2d-2}&0\\
    0&0&\omega^d
  \end{pmatrix}$
  of a cyclic group $\mathbb{Z}/2d\mathbb{Z}$ where
  $\omega = e^{{2\pi i}/2d}$ is a primitive $2d$-root of unity. Then
  consider the form
  $H=
  \prod^{2d-1}_{j=0}(\omega^{2j}x+\omega^{(2d-2)j}y+\omega^{dj}z)$.
  We have that
  \begin{align*}
    H&=\prod^{2d-1}_{j=0}(\xi^{j}x+\xi^{-j}y+(-1)^{j}z)\\
     & = \prod^{d-1}_{j=0}(\xi^{j}x+\xi^{-j}y+(-1)^{j}z)(\xi^{j+d}x+\xi^{-j+d}y+(-1)^{j+d}z)\\
     & = \prod^{d-1}_{j=0}(\xi^{j}x+\xi^{-j}y+(-1)^{j}z)(\xi^{j}x+\xi^{-j}y+(-1)^{j+d}z)\\
     & = \prod^{d-1}_{j=0}(\xi^{j}x+\xi^{-j}y-z)(\xi^{j}x+\xi^{-j}y+z)=F.
  \end{align*}
  Note that the monomials fixed by the action of $M_{2,2d-2,d}$,
  $M_{0,2d-4,d-2}$ and $M_{0,1,a}$ are the same, where
  $(2d-4)a=(d-2)$, since $d$ is odd such integer $a$ exists. By
  Theorem \ref{galois mons} we get that the number of monomials fixed
  by any of those actions is $d+3$. On the other hand Theorem 2, in
  \cite{LWW} implies that the number of terms with non-zero
  coefficient in $H$ and then in $F$ is exactly $d+3$ which implies
  that $\mu(I)=d+3$.

  Now assume that $d=2k$ and consider the action of
  $M_{2,2d-2,0}=\begin{pmatrix}
    \omega^2 &0& 0\\
    0 &\omega^{2d-2}&0\\
    0&0&1
  \end{pmatrix}$
  of a cyclic group $\mathbb{Z}/2d\mathbb{Z}$. Consider the form
  $G=\prod^{2d-1}_{j=0}(\omega^{2j}x+\omega^{(2d-2)j}y+z)$ then we
  have
  \begin{align*}
    G &= \prod^{2d-1}_{j=0}(\xi^{j}x+\xi^{-j}y+z)\\
      & = \prod^{d-1}_{j=0}(\xi^{j}x+\xi^{-j}y+z)(\xi^{j+d}x+\xi^{-j+d}y+z)\\
      & =  \prod^{d-1}_{j=0}(\xi^{j}x+\xi^{-j}y+z)(-\xi^{j}x-\xi^{-j}y+z)\\
      & = (-1)^d\prod^{d-1}_{j=0}(\xi^{j}x+\xi^{-j}y+z)(\xi^{j}x+\xi^{-j}y-z)=F\\
  \end{align*} 
  also we have that $F=G=(\prod^{d-1}_{j=0}(\xi^{j}x+\xi^{-j}y+z))^2$
  and denote $f:=\prod^{d-1}_{j=0}(\xi^{j}x+\xi^{-j}y+z)$. Theorem $2$
  in \cite{LWW}, implies that the monomials in $f$ with non-zero
  coefficients are exactly the monomials of degree $d$ fixed by the
  action $M_{1,d-1,0}=\begin{pmatrix}
    \xi &0& 0\\
    0 &\xi^{d-1}&0\\
    0&0&1
  \end{pmatrix}$
  of a cyclic group $\mathbb{Z}/d\mathbb{Z}$. Therefore, using Theorem
  \ref{galois mons} we get that there are $3+d/2$ monomials with
  non-zero coefficients in $f$, and they are exactly the monomials of
  the form $(xy)^\alpha z^{d-2\alpha}$ and $x^d$ and $y^d$.

  We now count the monomials in $F=f^2$. First we claim that the form
  $f$ has alternating sign in the variable $z$. To show this we
  evaluate the form in $x=y=1$ then we get
$$
\prod^{d-1}_{j=0}(\xi^{j}+\xi^{-j}+z)=\prod^{d/2-1}_{j=0}(\xi^{j}+\xi^{-j}+z)(\xi^{j+d/2}+\xi^{-j+d/2}+z)
= (-1)^{d/2}\prod^{d/2-1}_{j=0}(w^2+a^2_j)
$$ where $a_j=\xi^{j}+\xi^{-j}$ and $z^2=-w$. Then this expression proves the claim. 

Multiplying $x^d$ and $y^d$ in $f$ with $d/2+1$ monomials
$(xy)^\alpha z^{d-2\alpha}$ gives $2(d/2+1)=d+2$ monomials in
$F$. Using the claim above we get that all $d+1$ monomials of degree
$2d$ of the form $(xy)^\beta z^{2d-2\beta}$ have non-zero coefficients
in $F$. Adding 2 corresponding to the monomials $x^{2d}$ and $y^{2d}$
we get that there are exactly $2d+5$ monomials in $F$ with non-zero
coefficients or equivalently $\mu(I)=2d+5$.
\end{proof}
\begin{prop}\label{dih}
  For integer $d\geq 2$ let $I\subset \mathbb{K}[x,y,x]$ be the ideal
  generated by all the monomials with non-zero coefficients in
  $F=\prod^{d-1}_{j=0}(\xi^jx+\xi^{-j}y+z)(\xi^{j}x+\xi^{-j}y-z)$,
  introduced in Proposition \ref{dihedral}. Then $I$ fails WLP from
  degree $2d-1$ to degree $2d$.
\end{prop}
\begin{proof}
  Suppose that $d=2k+1$, using Proposition \ref{dihedral} we have that
$$H_{S/I}(2d)=H_{S}(2d)-\mu(I)=(2d^2+3d+1)-(d+3)=2(d^2+d-1)> d(2d+1)=H_{S/I}(2d-1).$$ 
Consider the form
$K=
(x+y-z)\prod^{d-1}_{i=1}(\xi^ix+\xi^{-i}y+z)(\xi^{-i}x+\xi^{i}y-z)$
of degree $2d-1$. Since we have $(x+y+z)K=F$, the map
$\times (x+y+z) :(S/I)_{2d-1}\longrightarrow (S/I)_{2d}$ is not
injective.

Now assume $d=2k$, then Proposition \ref{dihedral} implies that
$$
H_{S/I}(2d)=H_S(2d)-\mu(I)=(2d^2+3d+1)-(2d+5)=2d^2+d-4< d(2d+1) =
H_{S/I}(2d-1).
$$
Therefore, in order to prove $S/I$ fails the WLP we need to prove that
the multiplication map by $x+y+z$ on the algebra from degree $2d-1$ to
degree $2d$ is not surjective. To do so we use the representation
theory of the symmetric group $S_2$ where $1$ acts trivially and $-1$
interchanges $x$ and $y$. Note that this action fixes the form
$x+y+z$. We look at the multiplicity of the alternating representation
of $S/I$ in degree $2d-1$ and $2d$. In degree $2d-1$ there are
$d(2d+1)$ monomials in $S/I$ that are fixed by the identity
permutation and there are $d$ monomials of the form
$(xy)^\alpha z^{2d-1-2\alpha}$ that are fixed by interchanging $x$ and
$y$. Therefore the multiplicity of the alternating representation in
degree $2d-1$ is $(d(2d+1) - d)/{2}=d^2$. In degree $2d$ there are
$(2d+1)(d+1)-(2d+5)=2d^2+d-4$ monomials in $S/I$ that are all fixed by
the identity permutation and there is no monomial of the form
$(xy)^\alpha z^{2d-2\alpha}$ in the algebra which is fixed by
interchanging $x$ and $y$, since they all belong to $I$. So the
multiplicity of the alternating representation of $S/I$ in degree $2d$
is $(2d^2+d-4)/2$. Since $(2d^2+d-4)/2> d^2$ for $d\geq 5$ the
multiplication by $x+y+z$ cannot be surjective by Schur's lemma.\\ For
$d=4$ computations in Macaulay2 show the multiplication map by
$x+y+z$ is not surjective on $S/I$ from degree $7$ to degree $8$.
\end{proof}
\begin{rem}
  In Proposition \ref{dih}, we have proved that for odd integer $d$
  the monomial ideals generated by the monomials of degree $2d$ with
  non-zero coefficients in $F$ fail WLP by failing injectivity in
  degree $2d-1$, therefore such ideals define minimal monomial
  Togliatti systems.
\end{rem}

\section{Acknowledgment}
This work was supported by the grant VR2013-4545. Computations using
the algebra software Macaulay2 \cite{13} were essential to get the
ideas behind some of the proofs.

\bibliography{mybib}{}
\bibliographystyle{plain}
\end{document}